\documentclass{article}
\usepackage{amssymb}
\usepackage{fancybox}
\usepackage{tikz}
\usepackage{graphicx}
\usepackage{verbatim}
\usepackage{subfigure}
\usepackage{amsthm}

\usepackage{amsmath}  
\usepackage{bookmark}
\usepackage{mathrsfs} 
\usepackage{geometry}
\usepackage{esint}
\numberwithin{equation}{section}
\usepackage{cite}
\newtheorem{thm}{Theorem}[section]
\newtheorem{prop}[thm]{Proposition}
\newtheorem{lem}[thm]{Lemma}

\theoremstyle{remark}
\newtheorem{rem}[thm]{Remark}

\theoremstyle{definition}
\newtheorem{defn}{Definition}[section]
\renewcommand\thefootnote{\ding{\numexpr171+\value{footnote}}}
\begin{document}
	\title{Weak solutions of the three-dimensional hypoviscous elastodynamics with finite kinetic energy}
	\author{Ke Chen\footnotemark[1] \and Jie Liu\footnotemark[2]}
	\renewcommand{\thefootnote}{\fnsymbol{footnote}}
	\footnotetext[1]{School of Mathematical Sciences, Fudan University, Shanghai 200433,  China. 
		Email: kchen18@fudan.edu.cn.}
	\footnotetext[2]{School of Mathematical Sciences, Fudan University, Shanghai 200433,  China. 
		Email: j\_liu18@fudan.edu.cn.}
	\maketitle
	\begin{abstract}
		We construct weak solutions to the 3D hypoviscous incompressible elastodynamics with finite kinetic energy which was unknown in literatures. Our result holds for fractional hypoviscosity $(-\Delta)^\theta$, where $0\leq\theta<1$. The proof {consists of a convex integration scheme with new building blocks of 2D intermittency and suitable temporal correctors, which are motivated by} the inherent geometric structure of the viscoelastic equations.
	\end{abstract}
	\textbf{Keywords:} Viscoelastic Fluid; Convex Integration; Weak Solution
	
	\section{Introduction}
	
	We consider the following Oldroyd-B system describing incompressible viscoelastic fluids posed on the periodic box $\mathbb{T}^3$:
	\begin{equation}\label{elastic1}
		\left\{
		\begin{aligned}
			&\partial_tv+v\cdot\nabla v+(-\Delta)^\theta v+\nabla p=\operatorname{div}(FF^T),\\
			&\partial_tF+v\cdot\nabla F=\nabla v F,\\
			&\operatorname{div} v=0, ~\operatorname{div} F^T=0.
		\end{aligned}
		\right.
	\end{equation}
	Here $\theta\in[0,1)$ is a given constant, $v(t,x)$ represents the velocity field of fluids, $p(t,x)$ the scalar pressure, and $F(t,x)$ the deformation gradient tensor. We focus on solutions with the following normalization:
	\begin{equation*}
		\int_{\mathbb{T}^3}v dx=0,\quad\quad\quad \int_{\mathbb{T}^3}(F-\mathrm{Id}) dx=0,
	\end{equation*}
	where $\mathrm{Id}$ is the identity matrix. Throughout the paper we will adopt the notations of
	\begin{equation*}
		(\nabla v)_{ij}=\frac{\partial v_i}{\partial x_j}, ~~~~~~
		(\nabla\cdot F)_i=\partial_jF_{ij}.
	\end{equation*}
	Here and in what follows, we use the summation convention over repeated indices.
	
	Our goal is to construct weak solutions of \eqref{elastic1}. We first introduce the notion of weak solutions that we consider in this paper as follows:
	\begin{defn}\label{definition}
		Let $T>0$, we say $(v,F)\in C([0,T];L^2(\mathbb{T}^3))$ is a weak solution of system \eqref{elastic1} if $v$ and $F^T$ are weakly divergence free, and they satisfy \eqref{elastic1} in the sense of distributions, that's to say,
		\begin{equation*}
			\begin{aligned}
				&\int_{0}^{T} \int_{\mathbb{T}^{3}} \partial_{t} \psi \cdot v+\nabla \psi :(v \otimes v-FF^T) -(-\Delta)^\theta \psi \cdot v d x d t=0, \\
				&\int_{0}^{T} \int_{\mathbb{T}^{3}} \partial_{t} \psi \cdot F^k+\nabla \psi :(F^k \otimes v-v \otimes F^k) d x d t=0, \quad\quad k=1,2,3,
			\end{aligned}
		\end{equation*}
		hold for any divergence free function $\psi\in C_0^\infty((0,T)\times \mathbb{T}^3)$, where $F^k (k=1,2,3)$ is the $k$-th column vector of $F$.
	\end{defn}
	
	Before proceeding, we mention some famous works in the elastic and viscoelastic fluid  equations. In \cite{Siderisglobal}, Sideris and Thomases proved the global well-posedness of classical solutions to the 3D incompressible isotropic elastodynamics with small initial displacements. For 2D incompressible isotropic elastodynamics, the almost global existence  was proved in \cite{LeiAlmost} for small initial data. By observing an improved null structure for the system in Lagrangian coordinates, Lei proved the global well-posedness of 2D incompressible elastodynamics in \cite{LeiGlobal}. 
	As for viscoelastic fluid, Chemin and Masmoudi \cite{CheminAbout2001} proved the existence of {local and global small solutions in critical Besov spaces. 
		The }global well-posedness near equilibrium was first obtained in \cite{LinOn} for the two-dimensional case. Lei and Zhou \cite{LeiGlobal2005} obtained similar results through an incompressible limit process working directly on the deformation tensor $F$. More recently, Hu and Lin \cite{Huglobal} gave a global weak solution with discontinuous initial data in 2D. Zhu \cite{ZhuGlobal} proved the global existence of small solutions without any physical structure in 3D. We also refer to \cite{ChenGlobal,  GuillopeGlobal, ElgindiGlobal,LeiGlobal2008,Lions} for related results.
	To our knowledge, there is no result for existence of global weak solutions when the initial data is not small. In this paper, we construct infinitely many weak solutions with finite energy, while the initial data can be arbitrarily large. Theorem \ref{main thm} also demonstrates the nonuniqueness of $C_tL_x^2$ weak solutions to the Cauchy problem of \eqref{elastic1}, {where the initial data is satisfied in $L^2$ space.}
	
	The main result is stated as follows:
	\begin{thm}
		Suppose $\theta\in[0,1)$ is a given constant. For any {$T>0$,} there exist non-trivial weak solutions $(v,F)\in C([0,T];H^\beta(\mathbb{T}^3))$ of \eqref{elastic1} in the sense of Definition \ref{definition}, where the constant $\beta>0$ depends only on $\theta$.
	\end{thm}
	More precisely, we have the following result of h-principle type:
	\begin{thm}\label{main thm}
		For any given $\theta\in[0,1)$ and {$T>0$,} suppose $u$ is a vector field and $G$ is a tensor which are smooth on $[0,T]\times\mathbb{T}^3$ and satisfy
		\begin{equation*}
			\begin{aligned}
				\operatorname{div}u=0,~~~&\operatorname{div}G^T=0,\\
				\int_{\mathbb{T}^3}udx=0,~~~&\int_{\mathbb{T}^3}(G-\mathrm{Id})dx=0.
			\end{aligned}
		\end{equation*}
		Then, for any fixed $\varepsilon>0$, there exists a weak solution $(v,F)\in C([0,T];H^\beta(\mathbb{T}^3))$ of \eqref{elastic1} in the sense of Definition \ref{definition}, satisfying
		\begin{align}
			\label{thmL1}&\|v-u\|_{L_t^\infty L_x^1}+\|F-G\|_{L_t^\infty L_x^1}\leq\varepsilon,\\
			&\label{thmsupport}\operatorname{supp}_t(v,F)\subseteq O_\varepsilon (\operatorname{supp}_t(u,G)),
		\end{align}
		where the constant $\beta>0$ depends only on $\theta$, and we denote $\operatorname{supp}_t(v,F)=\operatorname{supp}_tv\cup\operatorname{supp}_tF$ and 
		\begin{equation*}
			O_\varepsilon(S):=\{t\in[0,T]: \exists s\in S,\ \text{such that }\  |s-t|\leq\varepsilon\}.
		\end{equation*}
	\end{thm}
	\begin{rem}
		By Theorem \ref{main thm}, system \eqref{elastic1} admits nontrivial weak solutions with compact temporal supports. This implies the non-uniqueness of weak solutions to \eqref{elastic1} in the sense of Definition \ref{definition}.
	\end{rem}
	When $\theta=0$, the system \eqref{elastic1} is referred as incompressible elastodynamics, which has similar structure {with the following 3D ideal magnetohydrodynamic (MHD) equations:}
	\begin{equation*}\
		\left\{
		\begin{aligned}
			&\partial_tu+\operatorname{div}(u\otimes u-B\otimes B)+\nabla p=0,\\
			&\partial_tB+\operatorname{div}(B\otimes u -u\otimes B)=0,\\
			&\operatorname{div} u=\operatorname{div} B=0,
		\end{aligned}
		\right.
	\end{equation*}
	where $u$ is the velocity field and $B$ the magnetic field. {Actually, let $F^k (k=1,2,3)$ denote the $k$-th column vector of $F$, then we can reformulate \eqref{elastic1} as
		\begin{equation}\label{elastic2}
			\left\{
			\begin{aligned}
				&\partial_tv+\operatorname{div}(v\otimes v-F^k\otimes F^k)+(-\Delta)^\theta v+\nabla p=0,\\
				&\partial_tF^i+\operatorname{div}(F^i\otimes v -v\otimes F^i)=0,\ \  i=1,2,3,\\
				&\operatorname{div} v=\operatorname{div} F^i=0.
			\end{aligned}
			\right.
		\end{equation}
		We find that even though the deformation tensor is a matrix, each column vector $F^i$ obeys similar equations with the magnetic field $B$ in MHD equations.}
	In \cite{Beekieweak}, Beekie et al. applied the convex integration scheme to construct $C_tL_x^2$ weak solutions to the ideal MHD equations, {using intermittent shear flows as building blocks. The shear flows have 1D spatial intermittency and thus only permits a viscosity term $(-\Delta)^\theta u$ with $\theta<\frac{3}{4}$.} In this paper, our scheme works for  $\theta<1$ by constructing new building blocks with 2D spatial intermittency {and introducing suitable temporal correctors}. Specially, our method can also be applied to MHD equations. Recently, after our work, Li et al. \cite{LZZ22} extended this to the full Lions exponent $\theta < \frac 54$ for viscous and resistive MHD equations by introducing an extra temporal intermittency.
	
	
	{The proof of Theorem \ref{main thm} builds on the scheme of} convex integration, which can be traced back to the work of Nash \cite{NashC1}. In \cite{NashC1}, Nash used an iteration scheme to construct $C^1$ isometric embeddings. Then, De Lellis and Sz\'{e}kelyhidi Jr. applied Nash's idea to fluid dynamics and devised in {\cite{LellisEuler,LellisDissi}} a ``convex integration" scheme leading to continuous dissipative solutions of the Euler equation, which was a significant step to solve the flexible part of Onsager's conjecture.
	Subsequently, after a series of advancements \cite{BuckAnomalous, IsettOn,Lellisdissi, BuckOnsager}, the Onsager's conjecture was finally resolved by Isett in \cite{IsettA}, {using a key ingredient by Daneri and
		Sz\'{e}kelyhidi Jr. \cite{DS17}. Then, Buckmaster et al. \cite{BDSV18} proved the same results for dissipative solutions. For the rigid part of Onsager's conjecture, one can refer to \cite{E94, ConstantinOnsager}.} The scheme has also been extended and adapted to various problems in mathematical physics, see \cite{CordobaLack, DaiNon, Beekieweak, BuckConvex, camillosurvey}  and references therein. 
	
	Furthermore, the convex integration techniques have fundamental analogies with the theory of turbulence, and features of turbulent flows (such as intermittency) have inspired researchers to develop and extend the convex integration constructions. Recently, the method was extended to the Navier-Stokes equations in \cite{bucknonuniqueness}, by constructing 3D intermittent Beltrami flows to treat the dissipative term. And more intermittent building blocks were adapted in the convex integration scheme, such as intermittent jets \cite{buckwild} and intermittent Mikado flows \cite{ModenaNon}. 
	
	{In this paper, we construct a new building block with 2D intermittency to adapt to the inherent geometric structure of viscoelastic equations \eqref{elastic2}. More precisely, we introduce the intermittent velocity and deformation flows of the following forms:
		$$W_{\xi} = \phi_{\xi}\varphi_{\xi_1}\xi, \quad  E_{\xi} = \phi_{\xi}\varphi_{\xi_1}\xi_2,$$
		where $(\xi,\xi_1,\xi_2)\subset \mathbb{Q}^3$ is an orthonormal basis, and $\phi_{\xi}, \varphi_{\xi_1}$  are spatial concentration functions with one oscillation direction $\xi,\xi_1$ respectively, see Section 2.3 for details. As mentioned in \cite{Beekieweak}, the previous building blocks (intermittent Beltrami flows, intermittent jets, respectively viscous eddies), are not applicable to system \eqref{elastic2}. The geometry of the nonlinear terms of system \eqref{elastic2} requires the building blocks’ direction of oscillation to be orthogonal to two direction vectors, only permitting the usage of 1D intermittency. To overcome this difficulty, we introduce two new types of temporal correctors to cancel the high spatial oscillations of nonlinear terms. }
	
	{The remainder of this paper} is organized as follows. In Section \ref{scheme} we introduce the convex integration scheme, present the main iteration scheme and give a proof of the Theorem \ref{main thm}. In Section \ref{pert} we define the perturbations and show the main estimates. Finally, we estimate the new stresses parts by parts and complete the iteration in Section \ref{eststress}.
	
	\paragraph{Notations:} Throughout the paper we use the following notations: $$\|f\|_{L^p}=\|f\|_{L_t^\infty L_x^p} \text{~for~} 1\leq p\leq\infty,$$ $$\|f\|_{C_{x,t}^N}=\sum_{0\leq n+|\zeta|\leq N}\|\partial_t^n D^\zeta f\|_{L^\infty},$$ 
	and 
	$$\|f\|_{W^{s,p}}=\|f\|_{L_t^\infty W_x^{s,p}} = \sum_{0\le |\zeta| \le s} \|D^\zeta f\|_{L_t^\infty L_x^{p}} ,$$ 
	where $\zeta=(\zeta_1,\zeta_2,\zeta_3)$ is the multi-index and $D^\zeta=\partial_{x_1}^{\zeta_1}\partial_{x_2}^{\zeta_2}\partial_{x_3}^{\zeta_3}$.
	
	For $f\in L_x^2(\mathbb{T}^3)$, we define the average integral operator
	$$\fint_{\mathbb{T}^3} f dx := \frac{1}{(2\pi)^3} \int f dx,$$
	and  the projection operator $$ \mathbb{P}_{\neq 0} f=f-\fint_{\mathbb{T}^{3}} f,$$  which projects a function onto its nonzero frequencies. We will write $A\lesssim B$ to denote that there exists a constant $C>0$ such that $A\leq CB$.
	
	\section{Convex Integration Scheme}\label{scheme}
	
	For every integer $q> 0$, we will construct a solution $(v_q,R^v_q,F^1_q,R^1_q,F^2_q,R^2_q,F^3_q,R^3_q)$ to the approximation system
	\begin{equation}\label{app elastic}
		\left\{\begin{aligned}
			&\partial_tv_q+\operatorname{div}(v_q\otimes v_q-F^k_q\otimes F^k_q)+(-\Delta)^\theta v_q+\nabla p_q=\operatorname{div}R^v_q,\\
			&\partial_tF^i_q+\operatorname{div}(F^i_q\otimes v_q-v_q\otimes F^i_q)=\operatorname{div}R^i_q,\quad\  i=1,2,3,\\
			&\operatorname{div} v_q=\operatorname{div} F^i_q=0,
		\end{aligned}\right.
	\end{equation}
	where the Reynolds stress $R^v_q$ is assumed to be a trace-free symmetric matrix, and $R^i_q$ is assumed to be skew-symmetric. {Besides,} $\|R^v_q\|_{L^1}$ and $\|R^i_q\|_{L^1}$ go to zero as $q\rightarrow \infty$, and $(v_q,F^1_q,F^2_q,F^3_q)$ will converge to a weak solution of \eqref{elastic2}.
	
	For $\theta\in [0,1)$ given in system \eqref{elastic1}, we denote
	\begin{equation}\label{theta}
		\theta_\ast=\left\{
		\begin{aligned}
			& 2\theta-1, & \frac{1}{2}<\theta<1; \\
			& 0, & 0\leq\theta\leq\frac{1}{2}.
		\end{aligned}
		\right.
	\end{equation}
	Then, we shall {fix} a parameter $\alpha { \in \mathbb{Q}}$ satisfying
	\begin{equation}\label{eq:alpha}
		0<\alpha\leq\frac{1-\theta_\ast}{8}\in \left(0,\operatorname{min}\left\{\frac{1-\theta}{4},\frac{1}{8}\right\}\right].
	\end{equation}
	In order to quantify the convergence of the stresses, we introduce a frequency parameter $\lambda_q$ and an amplitude parameter $\delta_q$ defined as follows:
	\begin{equation}\label{defpara}
		\lambda_q=a^{b^q},~~~~~~\delta_q=\lambda_1^{3\beta}\lambda_q^{-2\beta},
	\end{equation}
	where a large parameter $b\in\mathbb{N}$ and a small parameter $\beta>0$ would be {fixed} such that
	\begin{equation}\label{para ineq}
		800<\alpha b\in \mathbb{N} , ~~~\beta b^2\leq \frac{1}{16},
	\end{equation}
	and $a\gg 1$ will be chosen as a sufficiently large multiple of a geometric constant $N_\Lambda\in\mathbb{N}$ (which is fixed in Remark \ref{geometric const}).
	
	By induction, we will assume the following estimates on the solution of \eqref{app elastic} at level $q$:
	\begin{align}
		\label{STRESS L1}&\sum_{i=1}^3\|R^i_q\|_{L^1}\leq \delta_{q+1},~~~\|R^v_q\|_{L^1}\leq \delta_{q+1},\\
		\label{STRESS C1}&\|R^i_q\|_{C_{x,t}^1}+\|R^v_q\|_{C_{x,t}^1}\leq \lambda_{q}^{10},\\
		\label{C1norm}&\sum_{i=1}^3\|F^i_q\|_{C_{x,t}^1}+\|v_q\|_{C_{x,t}^1}\leq \lambda_{q}^4.
	\end{align}
	\begin{prop}[Main Iteration]\label{main iteration}
		There exist a universal constant $M$ and a sufficiently large parameter $a=a(\theta, b, \beta)$ such that the following holds: let $(v_q, R_q^v, F^i_q, R_q^i)$ be a solution of \eqref{app elastic} satisfying the inductive estimates \eqref{STRESS L1}-\eqref{C1norm}, {then there exist functions $(v_{q+1}, R_{q+1}^v, F^i_{q+1}, R_{q+1}^i)$ solving \eqref{app elastic} and satisfying \eqref{STRESS L1}-\eqref{C1norm} with $q$ replaced by $q+1$. Furthermore, we have}
		\begin{align}\label{iteration L2}
			&\|v_{q+1}-v_q\|_{L^2}+\sum_{i=1}^3\|F^i_{q+1}-F^i_q\|_{L^2}\leq M\delta_{q+1}^\frac{1}{2},\\
			&\|v_{q+1}-v_q\|_{L^1}+\sum_{i=1}^3\|F^i_{q+1}-F^i_q\|_{L^1}\leq \delta_{q+2},\label{iteration L1}\\
			&\bigcup_{i=1}^3\operatorname{supp}_t(v_{q+1},R^v_{q+1},F^{i}_{q+1},R^i_{q+1})\subseteq \bigcup_{i=1}^3O_{\delta_{q+2}}\left(\operatorname{supp}_t(v_{q},R^v_{q},F^{i}_{q},R^i_{q})\right).\label{itersupp}
		\end{align}
	\end{prop}
	\subsection{Proof of Theorem \ref{main thm}}
	{Following \cite{Beekieweak}, we introduce the symmetric inverse divergence operator $\mathcal{R}$ and skew-symmetric inverse divergence operator $\mathcal{R}^F$ by
		\begin{equation*}
			\begin{aligned}
				&(\mathcal{R} v)_{k l}=\partial_{k} \Delta^{-1} v^{l}+\partial_{l} \Delta^{-1} v^{k}-\frac{1}{2}\left(\delta_{k l}+\partial_{k} \partial_{l} \Delta^{-1}\right) \operatorname{div} \Delta^{-1} v,\\
				&(\mathcal{R}^Ff)_{ij}:=\epsilon_{ijk}(-\Delta)^{-1}(\operatorname{curl}f)_k,
			\end{aligned}
		\end{equation*}
		where $i,j,k,l\in\{1,2,3\}$, $\epsilon_{ijk}$ is the Levi-Civita tensor, and functions $v,f: \mathbb{R}^3\rightarrow\mathbb{R}^3$ satisfies $\int_{\mathbb{T}^3} v=0$ and $\operatorname{div}f=0$ respectively. By a direct calculation we see that $\operatorname{div}\mathcal{R}(v)=v$ and $\operatorname{div}\mathcal{R}^F(f)=f$. The operator $\mathcal{R}$ returns a symmetric, trace-free matrix, and $\mathcal{R}^F$ returns a skew-symmetric matrix. 
		By standard Calderon-Zygmund and Schauder estimates we have
		\begin{equation}\label{est antidiv}
			\begin{aligned}
				&\|\mathcal{R}\|_{L^{p} \rightarrow W^{1, p}} \lesssim 1, &\quad\|\mathcal{R}\|_{C^{0} \rightarrow C^{0}} \lesssim 1, &\quad\left\|\mathcal{R} \mathbb{P}_{\not=0} u\right\|_{L^{p}} \lesssim\left\||\nabla|^{-1} \mathbb{P}_{\neq 0} u\right\|_{L^{p}},
			\end{aligned}
		\end{equation}
		for $1<p<\infty$. 
		The above estimates also hold for $\mathcal{R}^F$.}
	
	Now we turn to the proof of Theorem \ref{main thm}.	For $q=0$, take $v_0=u$, $F^i_0=G^i$, and define $p_0, R_0^v, R_0^i$ as
	\begin{equation*}
		\begin{aligned}
			p_0&={-\frac{1}{3}}|v_0|^2+{\frac{1}{3}}\sum_{i=1}^3|F^i_0|^2,\\
			R_0^v&=\mathcal{R}\left(\partial_t v_0+(-\Delta)^\theta v_0\right)+v_0\otimes v_0-F^i_0\otimes F^i_0 {+ p_0 \mathrm{Id}},\\
			R_0^i&=\mathcal{R}^F(\partial_t F^i_0)+F^i_0\otimes v_0- v_0\otimes F^i_0.\\
		\end{aligned}
	\end{equation*}
	Then for $a$ large enough, $(v_0, R_0^v, F^1_0,R^1_0,F^2_0,R^2_0,F^3_0,R^3_0)$ satisfies system \eqref{app elastic} and obeys the estimates \eqref{STRESS L1}-\eqref{C1norm}. For $q\geq 1$,  we inductively apply Proposition \ref{main iteration} to get a sequence of solutions $(v_q,R^v_q,F^1_q,R^1_q,F^2_q,R^2_q,F^3_q,R^3_q)$ satisfying the inductive estimates \eqref{STRESS L1}-\eqref{C1norm}. Using the bound \eqref{iteration L2} and interpolation, we obtain
	\begin{equation*}
		\begin{aligned}
			&\sum_{q\geq 0}\|v_{q+1}-v_q\|_{H^{\beta'}}+\sum_{q\geq 0}\sum_{i=1}^3\|F^i_{q+1}-F^i_q\|_{H^{\beta'}}\\
			&\quad\leq\sum_{q\geq 0}\|v_{q+1}-v_q\|_{L^2}^{1-\beta'}\|v_{q+1}-v_q\|_{H^1}^{\beta'}+\sum_{q\geq 0}\sum_{i=1}^3\|F^i_{q+1}-F^i_q\|_{L^2}^{1-\beta'}\|F^i_{q+1}-F^i_q\|_{H^1}^{\beta'}\\
			&\quad\lesssim\sum_{q\geq 0}\delta_{q+1}^\frac{1-\beta'}{2}\lambda_{q+1}^{4\beta'}={\lambda_1^{\frac{3\beta(1-\beta')}{2}}}\sum_{q\geq 0}\lambda_{q+1}^{-\beta(1-\beta')+4\beta'}\lesssim 1,
		\end{aligned}
	\end{equation*}
	for $\beta'<\frac{\beta}{4+\beta}$. Hence the solution sequence is Cauchy and there exists a limit $(v, F^1, F^2, F^3)=\lim_{q\rightarrow\infty}(v_q, F_q^1, F_q^2, F_q^3)$, which is a weak solution of \eqref{elastic2} because $\lim_{q\rightarrow\infty}R_q^v=\lim_{q\rightarrow\infty}R_q^i=0$ in $C([0,1]; L^{1}(\mathbb{T}^3))$.
	Finally, the estimate \eqref{thmL1},  \eqref{thmsupport} are  direct results of \eqref{iteration L1} and \eqref{itersupp}. This completes the proof of Theorem \ref{main thm}.
	\subsection{Geometric Lemmas}
	The idea of the construction of perturbations mainly comes from the following two lemmas, with which we can cancel the previous stress by the low frequency of quadratic terms. One may refer to \cite{Beekieweak} for the proof of the lemmas.
	\begin{lem}\label{Representation of symmetric matrices}(Representation of symmetric matrices)There exists a constant $\varepsilon_v>0$ and a finite set $\Lambda_v\subset\mathbb{S}^2\cap\mathbb{Q}^3$ consisting of vectors $\xi$ with associated orthonormal basis $(\xi,\xi_1,\xi_2)$, such that for each symmetric trace-free matrix $R\in B_{\varepsilon_v}(\mathrm{Id})$ we have the identity
		\begin{equation*}
			R=\frac{1}{2}\sum_{\xi\in \Lambda_v}(\gamma_\xi(R))^2(\xi\otimes\xi),
		\end{equation*}
		where $\gamma_\xi\in C^{\infty}(B_{\varepsilon_v}(\mathrm{Id})), \xi\in \Lambda_v$.
	\end{lem}
	\begin{lem}\label{Representation of skew-symmetric matrices}(Representation of skew-symmetric matrices) For every $N\in\mathbb{N}$, there exists a constant $\varepsilon_F>0$ and pairwise disjoint sets $\Lambda_i\subset\mathbb{S}^2\cap\mathbb{Q}^3$, $i=1,2,\ldots,N$, consisting of vectors $\xi$ with associated orthonormal basis $(\xi,\xi_1,\xi_2)$,  such that for each skew-symmetric matrix $R\in B_{\varepsilon_F}(0)$ we have the identity
		\begin{equation*}
			R=\sum_{\xi\in \Lambda_i}(\gamma_\xi(R))^2(\xi_2\otimes\xi-\xi\otimes\xi_2),
		\end{equation*}
		for each $i=1,2,\dots,N$. Here $\gamma_\xi\in C^{\infty}(B_{\varepsilon_F}(0)), \xi\in \Lambda_i$.
	\end{lem}
	\begin{rem}\label{geometric const}
		We choose $\Lambda_v$ and $\Lambda_i (i=1,2,3)$ such that there do not exist two parallel vectors in two different sets, and for two vectors $\xi\neq\xi'$,  the associated orthonormal bases satisfy $\xi_1\neq\xi'_1$. For convenience, we denote $\Lambda=\bigcup_{i=1}^3\Lambda_i\cup\Lambda_v$. By our choice, there exists a constant $N_\Lambda\in\mathbb{N}$ such that
		\begin{equation*}
			\{N_\Lambda\xi,N_\Lambda\xi_1,N_\Lambda\xi_2\}\subset \mathbb{Z}^3, ~\xi\in \Lambda.
		\end{equation*}
		
	\end{rem}
	\subsection{Intermittent Flow}
	Let $\Psi:\mathbb{R}\rightarrow\mathbb{R}$ be a smooth cutoff function supported on the interval $[-1,1]$. Assume it is normalized in such a way that $\phi:=-\frac{d^2}{dx^2}\Psi$ satisfies
	\begin{equation*}
		\int_\mathbb{R}\phi^2(x)dx=2\pi.
	\end{equation*}
	For parameters $0<\sigma \ll r \ll 1$, we define the rescaled functions
	\begin{equation*}
		\phi_{r}(x):=\frac{1}{r^{1/2}}\phi\left(\frac{x}{r}\right),~ \phi_{\sigma}(x):=\frac{1}{\sigma^{1/2}}\phi\left(\frac{x}{\sigma}\right), ~\Psi_{\sigma}(x):=\frac{1}{\sigma^{1/2}}\Psi\left(\frac{x}{\sigma}\right).
	\end{equation*}
	And we periodize the above functions so that we can view the resulting functions as functions defined on $\mathbb{T}$. We fix a large parameter $\lambda$ such that $\lambda \sigma\in \mathbb{N}$, and a large time oscillation parameter $\mu\gg\sigma^{-1}$. For every $\xi\in \Lambda$, define
	\begin{equation*}
		\begin{aligned}
			&\phi_{\xi}(t,x):=\phi_{\xi,r,\sigma,\lambda,\mu}(t,x)=\phi_{r}(\lambda \sigma N_\Lambda (\xi\cdot x+\mu t)),\\
			&\varphi_{\xi_1}(x):=\phi_{\xi_1,\sigma,\lambda}(x)=\phi_{\sigma}(\lambda \sigma N_\Lambda \xi_1\cdot x),\\
			&\Psi_{\xi_1}(x):=\Psi_{\xi_1,\sigma,\lambda}(x)=\Psi_{\sigma}(\lambda \sigma N_\Lambda \xi_1\cdot x),
		\end{aligned}
	\end{equation*}
	which are $\left(\frac{\mathbb{T}}{\lambda \sigma}\right)^3$ periodic. By definition, we have
	\begin{align}\label{phi t}
		-\Delta\Psi_{\xi_1}(x)=\lambda^2N^2_\Lambda\varphi_{\xi_1}(x),\quad\quad\text{and}\quad\quad
		\xi\cdot\nabla\phi_\xi=\frac{1}{\mu}\partial_t\phi_\xi.
	\end{align}
	
	The \textit{intermittent velocity flows} are defined by
	$$W_{\xi}(t,x) := \phi_{\xi}(t,x)\varphi_{\xi_1}(x)\xi, \quad \xi\in \Lambda,$$
	and the \textit{intermittent deformation flows} are defined by
	$$E_{\xi}(t,x) := \phi_{\xi}(t,x)\varphi_{\xi_1}(x)\xi_2, \quad \xi\in \bigcup_{i=1}^3\Lambda_i,$$
	Since the map $x \mapsto \lambda \sigma N_{\Lambda}\left(\xi \cdot x+\mu t, \xi_1 \cdot x,\xi_2 \cdot x\right)$ is the composition of a rotation by a rational orthogonal matrix mapping $\left\{e_{1}, e_{2}, e_{3}\right\}$ to $\left\{\xi, \xi_1, \xi_2\right\}$, a translation, and a rescaling by integers, we have
	\begin{align}\label{eq-block-L2}
		\int_{\mathbb{T}^3}\phi_{\xi}\varphi_{\xi_1} dx =0, \quad  \fint_{\mathbb{T}^3}\phi_{\xi}^2\varphi_{\xi_1}^2 dx = \frac{1}{(2\pi)^3}\int_{\mathbb{T}^3}\phi_{\xi}^2\varphi_{\xi_1}^2 dx = 1. 
	\end{align}
	Moreover, the following estimates hold:
	\begin{lem}\label{philemma}
		For any $1\leq p\leq \infty, M, N\in\mathbb{N}$, and $\xi\neq\xi'$ we have the following estimates
		\begin{align}
			&\|\nabla^M\partial_t^N\phi_\xi\|_{L^p}\lesssim(\lambda \sigma)^{M+N}r^{\frac{1}{p}-\frac{1}{2}-M-N}\mu^N,\label{phiparaLp}\\
			&\|\nabla^M\varphi_{\xi_1}\|_{L^p}+\|\nabla^M\Psi_{\xi_1}\|_{L^p}\lesssim\lambda^M\sigma^{\frac{1}{p}-\frac{1}{2}},\label{phiperpLp}\\
			&\|\nabla^M(\phi_\xi\varphi_{\xi_1})\|_{L^p}+\|\nabla^M(\phi_\xi\Psi_{\xi_1})\|_{L^p}\lesssim\lambda^Mr^{\frac{1}{p}-\frac{1}{2}}\sigma^{\frac{1}{p}-\frac{1}{2}},\label{phiproduct}\\
			&\|\phi_\xi\varphi_{\xi_1}\phi_{\xi'}\varphi_{\xi_1'}\|_{L^p}\lesssim \sigma^{\frac{2}{p}-1}r^{-1}\label{phippproduct},
		\end{align}
		where the implicit constants only depend on $p$, $N$ and $M$.
	\end{lem}
	\begin{proof}
		The first three inequalities follow from \cite{Beekieweak}. Here we only prove the last inequality. Recall that $\phi_\xi,\varphi_{\xi_1}, $ are $\left(\frac{\mathbb{T}}{\lambda \sigma}\right)^3$ periodic. In each box with side length $\frac{2\pi}{\lambda \sigma}$, the support of $\phi_\xi\varphi_{\xi_1}$ {consists of at most $2N_{\Lambda}$ parallel thin cubes} with length $\sim\lambda^{-1}$ (in the direction of $\xi_1$), width $\sim(\lambda{\sigma})^{-1}r$ (in the direction of $\xi$) and height $\sim(\lambda{\sigma})^{-1}$ (in the direction of $\xi_2$). {And in each box the support of $\phi_{\xi'}\varphi_{\xi_1'}$ consists of another family of parallel cubes  with the same size.
			Due to our specific choice of the set $\Lambda$ (see Remark \ref{geometric const}), the two families of support cubes are unparallel and the angel between any two different vectors $\xi\neq \xi'$ are larger than a universal constant $\theta_{\Lambda}>0$. In view of this, the intersections of two different support cubes are contained in much smaller cubes, with the length and width bounded by $\sim\lambda^{-1}$, and the height $\sim(\lambda\sigma)^{-1}. $ Moreover, since the number of such cubes are bounded by $4N_{\Lambda}^2$, the total size of the support  of $\phi_\xi\varphi_{\xi_1}\phi_{\xi'}\varphi_{\xi_1'} $ in each box with side length $\frac{2\pi}{\lambda\sigma}$ is bounded by $ \sim \frac{1}{\lambda^2}\frac{1}{\lambda \sigma}$.} We multiply the estimate by $(\lambda \sigma)^3$ to derive
		\begin{equation*}
			\begin{aligned}
				&|\operatorname{supp} \phi_\xi\varphi_{\xi_1}\phi_{\xi'}\varphi_{\xi_1'}|\lesssim{\sigma}^2,\\
				&\|\phi_\xi\varphi_{\xi_1}\phi_{\xi'}\varphi_{\xi_1'}\|_{L^1}\lesssim |\operatorname{supp} \phi_\xi\varphi_{\xi_1}\phi_{\xi'}\varphi_{\xi_1'}| \|\phi_\xi\varphi_{\xi_1}\phi_{\xi'}\varphi_{\xi_1'}\|_{L^\infty}\lesssim \frac{\sigma}{r}.
			\end{aligned}
		\end{equation*}
		Interpolating between $L^1$ and $L^\infty$ yields the estimate \eqref{phippproduct}. 
	\end{proof}
	Now we fix the values of the  parameters $\lambda$, $\sigma$, $r$ and $\mu$ as
	\begin{equation}\label{para def}
		\begin{aligned}
			&\lambda:=\lambda_{q+1}, &\sigma:=\lambda_{q+1}^{-(1-2\alpha)}, 
			&\  r:=\lambda_{q+1}^{-(1-6\alpha)}, &\mu:=\lambda_{q+1}^{1-\alpha}.
		\end{aligned}
	\end{equation}
	\section{The Perturbation}\label{pert}
	\subsection{Mollification}
	In order to avoid the loss of derivatives, we mollify the velocity and the deformation fields. Let $$\psi_\epsilon(t)=\epsilon^{-1}\psi\left(\frac{t}{\epsilon}\right), \quad\quad\quad\tilde \psi_\epsilon(x)=\epsilon^{-3}\tilde \psi\left(\frac{x}{\epsilon}\right),$$ be the  standard 1D and 3D Friedrichs mollifier sequences respectively, with $$\operatorname{supp}\psi\subseteq(-1,1),\quad\quad\quad \operatorname{supp}\tilde \psi\subseteq B_1(0).$$  Define a mollification of $v_q$, $F_q^i$, $R_q^v$, and $R_q^i$ $(i=1,2,3)$ in space and time at length scale $\ell$ by
	\begin{equation*}
		\begin{aligned}
			&v_{\ell} :=\left(v_{q} *_{x} \tilde \psi_{\ell}\right) *_{t} \psi_{\ell},
			&R^v_{\ell} :=\left(R^v_{q} *_{x} \tilde \psi_{\ell}\right) *_{t}\psi_{\ell},\\
			&F^i_{\ell} :=\left(F^i_{q} *_{x} \tilde \psi_{\ell}\right) *_{t}\psi_{\ell},
			&R^i_{\ell} :=\left(R^i_{q} *_{x} \tilde \psi_{\ell}\right) *_{t}\psi_{\ell}.\\
		\end{aligned}
	\end{equation*}
	Then one has 
	\begin{align}\label{suppmol}
		\bigcup_{i=1}^3\operatorname{supp}_t(v_{\ell},R^v_{\ell},F^{i}_{\ell},R^i_{\ell})\subseteq \bigcup_{i=1}^3O_{\ell}\left(\operatorname{supp}_t(v_{q},R^v_{q},F^{i}_{q},R^i_{q})\right).
	\end{align}
	By \eqref{app elastic} we obtain that $(v_{\ell},R^v_{\ell})$ and $(F^i_{\ell},R^i_{\ell})$ satisfy
	\begin{equation}\label{mollified elastic}
		\left\{
		\begin{aligned}
			&\partial_tv_{\ell}+\operatorname{div}(v_{\ell}\otimes v_{\ell}-F^k_{\ell}\otimes F^k_{\ell})+(-\Delta)^\theta v_\ell+\nabla p_{\ell}=\operatorname{div}\left(R^v_{\ell}+R^v_{comm}\right),\\
			&\partial_tF^i_{\ell}+\operatorname{div}(F^i_{\ell}\otimes v_{\ell}-v_{\ell}\otimes F^i_{\ell})=\operatorname{div}\left(R^i_{\ell}+R^i_{comm}\right), i=1,2,3,\\
			&\operatorname{div} v_{\ell}=\operatorname{div} F^i_{\ell}=0,
		\end{aligned}\right.
	\end{equation}
	where the traceless symmetric commutator stress $R^v_{comm}$ and the skew-symmetric stress $R^i_{comm}$ are given by
	\begin{equation}\label{Rcomm}
		\begin{aligned}
			{R}_{c o m m}^{v}=\left(v_{\ell} \hat{\otimes} v_{\ell}\right)-\left(F^i_{\ell}  \hat{\otimes} F^i_{\ell}\right)-\left(\left(v_{q} \hat{\otimes} v_{q}-F^i_{q} \hat{\otimes} F^i_{q}\right) *_{x} \tilde \psi_{\ell}\right) *_{t}\psi_{\ell}, \\
			{R}_{c o m m}^{i}=F^i_{\ell} \otimes v_{\ell}-v_{\ell} \otimes F^i_{\ell}-\left(\left(F^i_{q} \otimes v_{q}-v_{q} \otimes F^i_{q}\right) *_{x} \tilde \psi_{\ell}\right) *_{t}\psi_{\ell}.
		\end{aligned}
	\end{equation}
	Here we use $a\hat{\otimes}b$ to denote the traceless part of tensor $a\otimes b$.
	The new pressure $p_{\ell}$ is defined as
	\begin{equation*}
		p_{\ell}=\left(p_{q} *_{x} \tilde \psi_{\ell}\right) *_{t}\psi_{\ell}-{\frac{1}{3}}\left|v_{\ell}\right|^{2}+ {\frac{1}{3}}\sum_{i=1}^3\left|F^i_{\ell}\right|^{2} + {\frac{1}{3}} \left(\left(\left|v_{q}\right|^{2}-\sum_{i=1}^3\left|F^i_{q}\right|^{2}\right) *_{x} \tilde \psi_{\ell}\right) *_{t}\psi_{\ell}.
	\end{equation*}
	In view of \eqref{STRESS L1}, \eqref{STRESS C1} and \eqref{C1norm} we have for $M\in \mathbb{N}, N\in \mathbb{N}_+$,
	\begin{align}
		\label{Rl L1}&\left\|\nabla^MR^v_{\ell}\right\|_{L^1}+\sum_{i=1}^3\left\|\nabla^MR^i_{\ell}\right\|_{L^1}\lesssim\ell^{-M}\delta_{q+1},\\
		\label{Rl C1}&\left\|R^v_{\ell}\right\|_{C^N_{t,x}}+\sum_{i=1}^3\left\|R^i_{\ell}\right\|_{C^N_{t,x}}\lesssim\lambda_q^{10}\ell^{-N+1},\\
		\label{vl CN}&\left\|v_{\ell}\right\|_{C^N_{t,x}}+\sum_{i=1}^3\left\|F^i_{\ell}\right\|_{C^N_{t,x}}\lesssim\lambda_q^{4}\ell^{-N+1}.
	\end{align}
	For commutator stresses, we follow the estimate in \cite{ConstantinOnsager} and obtain
	\begin{align*}
		&\left\|{R}_{c o m m}^{i}\right\|_{L^1}\lesssim\left\|{R}_{c o m m}^{i}\right\|_{C^0_{t,x}}\lesssim\ell^2\|v_q\|_{C^1_{t,x}}\|F^i_{q}\|_{C^1_{t,x}}\lesssim \ell^2\lambda_q^4,\\
		&\left\|{R}_{c o m m}^{v}\right\|_{L^1}\lesssim\left\|{R}_{c o m m}^{v}\right\|_{C^0_{t,x}}\lesssim\ell^2(\|v_q\|_{C^1_{t,x}}^2+\sum_{i=1}^3\|F^i_{q}\|_{C^1_{t,x}}^2)\lesssim \ell^2\lambda_q^8.
	\end{align*}
	Now we fix the parameter $\ell$ as
	\begin{equation}\label{defell}
		\ell:=\lambda_q^{-20}.
	\end{equation}
	Recall that $\beta b^2\leq \frac{1}{16}$, hence 
	\begin{align}\label{commutator}
		\left\|{R}_{c o m m}^{v}\right\|_{L^1}+\sum_{i=1}^3\left\|{R}_{c o m m}^{i}\right\|_{L^1}{\lesssim}\delta_{q+2}.
	\end{align}
	\subsection{The Principal Part of Perturbations}
	Define the principal part of the perturbations as
	\begin{equation*}
		\begin{aligned}
			w_{q+1}^p{=\sum_{\xi\in\Lambda}a_\xi W_{\xi}}=\sum_{\xi\in\Lambda}a_\xi\phi_\xi\varphi_{\xi_1}\xi,\quad\quad\text{and}\quad\quad e_{q+1}^{i,p}{=\sum_{\xi\in\Lambda_i}a_\xi E_{\xi}}=\sum_{\xi\in\Lambda_i}a_\xi\phi_\xi\varphi_{\xi_1} \xi_2,
		\end{aligned}
	\end{equation*}
	where the amplitude functions $a_\xi$ are to be determined such that $R_q^v$ and $R_q^{i}$ can be cancelled by applying Lemma \ref{Representation of symmetric matrices} and Lemma \ref{Representation of skew-symmetric matrices}.  In order to cancel all the stresses, the velocity perturbation $w_{q+1}=v_{q+1}-v_\ell$ should contain wavevectors from $\Lambda=\bigcup_{i=1}^3\Lambda_i\cup\Lambda_v$. (Wavevectors from $\Lambda_v$ are used to cancel $R_\ell^v$, and wavevectors from $\Lambda_i$ take part in the cancellation of $R_\ell^i$.) The deformation perturbation $e_{q+1}^i=F_{q+1}^i-F_\ell^i$ should contain wavevectors from $\Lambda_i$.   To achieve this, we first introduce a smooth increasing function $\chi$ satisfying
	\begin{equation*}
		\chi(z)=\begin{cases}
			1,&0\leq z\leq 1;\\
			z,&z\geq 2.
		\end{cases}
	\end{equation*}
	And we define a temporal cut-off as in \cite{Luonon}: let $\Phi_q(t):[0,T]\to [0,1]$ be a smooth cut-off function that satisfies
	\begin{align}
		&\Phi_q(t)=1\ \ \text{on}\ \operatorname{supp}_t (R^v_\ell,R^1_\ell,R^2_\ell,R^3_\ell),\quad\quad\quad \ \operatorname{supp}\Phi_q\subset O_\ell\left(\operatorname{supp}_t (R^v_\ell,R^1_\ell,R^2_\ell,R^3_\ell)\right),\nonumber\\
		&\|\Phi_q\|_{C^k}\lesssim \ell^{-k}, \ \ k=1,2,3.\label{esPhi}
	\end{align} 
	Then we set
	\begin{equation*}
		\rho_{i}(t,x) :=2 \delta_{q+1} \varepsilon_{F}^{-1} \chi\left(\delta_{q+1}^{-1}|{R}_\ell^{i}(t,x)|\right), \quad i=1,2,3,
	\end{equation*}
	where $\varepsilon_F$ is the small radius in the geometric Lemma \ref{Representation of skew-symmetric matrices}. It is easy to verify that pointwise we have
	\begin{equation*}
		\left|\frac{R_\ell^{i}(t,x)}{\rho_{i}(t,x)}\right|=\left|\frac{R_\ell^{i}(t,x)}{2 \delta_{q+1} \varepsilon_{F}^{-1} \chi\left( \delta_{q+1}^{-1}|{R}_\ell^{i}(t,x)|\right)}\right| \leq \varepsilon_{F}.
	\end{equation*}
	By definition and \eqref{Rl L1} we have
	\begin{equation}\label{rhoi L1}
		\|\rho_i\|_{L^1}\lesssim\delta_{q+1}+\|R_\ell^{i}\|_{L^1}\lesssim \delta_{q+1},
	\end{equation}
	where the implicit constant only depends on $\varepsilon_{F}$. Furthermore, by standard H\"{o}lder estimates (see Proposition C.1 in \cite{BuckAnomalous}), \eqref{Rl C1}, and the parameter inequality $\ell\ll\delta_{q+1}$, we obtain
	\begin{align}
		&\|\rho_i\|_{C^0_{t,x}} \lesssim \ell^{-1}, \quad \left\|\rho_i^{\frac12}\right\|_{C^0_{t,x}} \lesssim \ell^{-1}, \quad  \left\|\rho_i^{-1}\right\|_{C^0_{t,x}} \lesssim \delta_{q+1}^{-1}, \label{eq:rhoC0}\\
		&\left\|\rho_i^\frac{1}{2}\right\|_{C^N_{t,x}}\lesssim \delta_{q+1}^{\frac{1}{2}}(\|\delta_{q+1}^{-1}R_\ell^i\|_{C^N_{t,x}}+\|\delta_{q+1}^{-1}R_\ell^i\|_{C^1_{t,x}}^N)\lesssim\delta_{q+1}^{-\frac{1}{2}}\ell^{-N}+\delta_{q+1}^{\frac{1}{2}-N}\ell^{-N}\lesssim\ell^{-2N} \label{rhoCj}\\
		& \left\|\rho_i^{-1}\right\|_{C^N_{t,x}} \lesssim \ell^{-N}\delta_{q+1}^{-N-1},\label{rho-1Cj}
	\end{align}
	for $N=1,2,3$. Then for each $i=1,2,3$, we define the deformation amplitude functions
	\begin{equation*}
		a_\xi(t,x):=a_{\xi,i}(t,x)=\rho_i^{\frac{1}{2}}\Phi_q(t)\gamma_\xi\left(\frac{-R_\ell^i}{\rho_i}\right), ~\xi\in\Lambda_i.
	\end{equation*}
	By Lemma \ref{Representation of skew-symmetric matrices} {and \eqref{eq-block-L2},} we have
	\begin{equation*}
		-{R}_\ell^{i}=\sum_{\xi\in\Lambda_i}a_\xi^2(\xi_{2} \otimes \xi-\xi \otimes \xi_{2})=\sum_{\xi\in\Lambda_i}a_\xi^2 \fint\phi_\xi^2\varphi_{\xi_1}^{2} (\xi_{2} \otimes \xi-\xi \otimes \xi_{2}).
	\end{equation*}
	Thus, we get
	\begin{equation}\label{cancelRi}
		\sum_{\xi \in \Lambda_{i}} a_{\xi}^{2} \phi_\xi^2\varphi_{\xi_1}^2\left(\xi_{2} \otimes \xi-\xi \otimes \xi_{2}\right)+{R}_\ell^{i}=\sum_{\xi \in \Lambda_{i}} a_{\xi}^{2} \mathbb{P}_{ \neq 0}\left(\phi_\xi^2\varphi_{\xi_1}^2\right)\left(\xi_{2} \otimes \xi-\xi \otimes \xi_{2}\right).
	\end{equation}
	Before giving the definition of the velocity amplitude functions, we first emphasize some factors that need to be considered. In order to cancel all the stresses, the velocity perturbation should have wavevectors from both $\Lambda_v$ and $\bigcup_{i=1}^3\Lambda_i$. The wavevectors in $\Lambda_v$ will be used to cancel the previous Reynolds stress $R^v_\ell$, as well as a new matrix $R_F$ which is related to wavevectors from $\bigcup_{i=1}^3\Lambda_i$:
	\begin{equation}\label{defRF}
		R_{F}=\sum_{i=1}^3\sum_{\xi\in\Lambda_i}a_\xi^2\left(\xi \otimes \xi-\xi_{2} \otimes \xi_{2}\right),
	\end{equation}
	see section \ref{secosc} for details. Note that the estimate of $L^1$ norm of $R_{F}$ is comparable to that of $R_q^v$ (see Lemma \ref{lemma estimate of a} below). It is too large to go into the new Reynolds stress so must be canceled together with $R_q^v$. This motivates us to define $\rho_v$ and the associated velocity amplitudes as
	\begin{equation*}
		\begin{aligned}
			\rho_{v}(t,x) &:=2 \delta_{q+1} \varepsilon_{v}^{-1} \chi\left(\delta_{q+1}^{-1}|{R}_\ell^{v}(t,x)+R_{F}|\right),\\
			a_\xi(t,x)&:=a_{\xi,v}(t,x)=\rho_v^{\frac{1}{2}}\Phi_q(t)\gamma_\xi\left(\mathrm{Id}-\frac{R_\ell^v+R_{F}}{\rho_v}\right), ~\xi\in\Lambda_v.
		\end{aligned}
	\end{equation*}
	Analogously we can verify that $\mathrm{Id}-\frac{R_\ell^v+R_{F}}{\rho_v}$ is in the domain of functions $\gamma_\xi$ {in Lemma \ref{Representation of symmetric matrices}} and the following {equality} holds:
	\begin{equation}\label{cancelRv}
		\sum_{\xi \in \Lambda_{v}} a_{\xi}^{2} \phi_\xi^2\varphi_{\xi_1}^2\left(\xi \otimes \xi\right)+{R}_\ell^{v}+R_{F}=\rho_v\mathrm{Id}+\sum_{\xi \in \Lambda_{v}} a_{\xi}^{2} \mathbb{P}_{ \neq 0}\left(\phi_\xi^2\varphi_{\xi_1}^2\right)\left(\xi \otimes \xi\right).
	\end{equation}
	We note that the definitions of $a_\xi$ for $\xi\in \Lambda_v$, respectively for $\xi\in\bigcup_{i=1}^3\Lambda_i$, are different. Throughout the paper we abuse this notation and write $a_\xi=a_{\xi,i}$ for $\xi\in\Lambda_i$ and $a_\xi=a_{\xi,v}$ for $\xi\in\Lambda_v$.
	
	Now we present the following estimates for the amplitude functions.
	\begin{lem}\label{lemma estimate of a}
		(Estimate of the amplitude functions) For any $\xi\in \Lambda=\bigcup_{i=1}^3\Lambda_i\cup \Lambda_v$  and $N=1,2,3$, we have the following estimates
		\begin{align}
			\label{aCN}&\|a_\xi\|_{C^N_{t,x}}\lesssim\ell^{-3N},\\
			\label{aL2}&\|a_\xi\|_{L^2}\lesssim\delta_{q+1}^\frac{1}{2}.
		\end{align}
	\end{lem}
	\begin{proof}We first consider $\xi\in\bigcup_{i=1}^3\Lambda_i$. By  \eqref{Rl C1} and \eqref{eq:rhoC0}-\eqref{rho-1Cj}, one gets that for $i=1,2,3$,
		\begin{align*}
			\|a_\xi\|_{C^N_{t,x}}\lesssim \|\rho_i^\frac{1}{2}\|_{C^N_{t,x}} + \|\rho_i^\frac{1}{2}\|_{C^0_{t,x}}\left( \|\Phi_q\|_{C^N}+\|\rho_i^{-1}R^i_\ell\|_{C^1_{t,x}}^N+\|\rho_i^{-1}R^i_\ell\|_{C^N_{t,x}}\right)\lesssim \ell^{-2N},\ \ \xi\in\Lambda_i.
		\end{align*}
		Moreover, \eqref{rhoi L1} yields that 
		\begin{align*}
			\|a_\xi\|_{L^2}\lesssim \|\rho_i\|^\frac{1}{2}_{L^1}\|\gamma_\xi\|_{C^0}\lesssim  \delta_{q+1}^\frac{1}{2},\ \ \xi\in\Lambda_i.
		\end{align*}
		It remains to prove the estimates for $\xi\in\Lambda_v$. Observe that 
		\begin{align*}
			\|R_F\|_{C^N_{t,x}}\lesssim\sum_{i=1}^3 \sum_{\xi\in\Lambda_i}\|a_\xi^2\|_{C^N_{t,x}}\lesssim \ell^{-2N},\ \ \ \ \|R_F\|_{L^1}\lesssim \sum_{i=1}^3 \sum_{\xi\in\Lambda_i}\|a_\xi\|_{L^2}^2\lesssim \delta_{q+1}.
		\end{align*}
		Using the same techniques used to derive \eqref{rhoi L1} and \eqref{rhoCj}, we obtain 
		\begin{equation*}
			\|\rho_v\|_{L^1}\lesssim \delta_{q+1},\ \ \ \|\rho_v^\frac{1}{2}\|_{C^  N_{t,x}}\lesssim \ell^{-\frac{5}{2}N}.
		\end{equation*}
		Hence we get 
		\begin{align*}
			& \|a_\xi\|_{C^N_{t,x}}\lesssim \|\rho_v^\frac{1}{2}\|_{C^N_{t,x}} + \|\rho_v^\frac{1}{2}\|_{C^0_{t,x}}\left( \|\Phi_q\|_{C^N}+\|\rho_v^{-1}(R^v_\ell+R_F)\|_{C^1_{t,x}}^N+\|\rho_v^{-1}(R^v_\ell+R_F)\|_{C^N_{t,x}}\right)\lesssim \ell^{-3N},\\
			&   \|a_\xi\|_{L^2}\lesssim \|\rho_v\|^\frac{1}{2}_{L^1}\|\gamma_\xi\|_{C^0}\lesssim  \delta_{q+1}^\frac{1}{2},\quad\quad\quad \ \xi\in\Lambda_v.
		\end{align*}
		This completes the proof.
	\end{proof}
	
	\subsection{Incompressibility Correctors}
	Notice that $w_{q+1}^p$ and $e_{q+1}^{i,p}$ are not divergence free, in order to fix it, we introduce the incompressibility correctors
	\begin{equation*}
		\begin{aligned}
			w_{q+1}^c:=&\frac{1}{N_\Lambda^2\lambda_{q+1}^2}\sum_{\xi\in\Lambda}\left(\operatorname{curl}(\nabla a_\xi\times(\phi_\xi\Psi_{\xi_1}\xi))+\nabla a_\xi\times\operatorname{curl}(\phi_\xi\Psi_{\xi_1}\xi)+a_\xi\nabla\phi_\xi\times\operatorname{curl}(\Psi_{\xi_1}\xi)\right),\\
			e_{q+1}^{i,c}:=&\frac{1}{N_\Lambda^2\lambda_{q+1}^2}\sum_{\xi\in\Lambda_i}\left(\operatorname{curl}(\nabla a_\xi\times(\phi_\xi\Psi_{\xi_1}\xi_2))+\nabla a_\xi\times\operatorname{curl}(\phi_\xi\Psi_{\xi_1}\xi_2){- a_\xi\Delta\phi_\xi\Psi_{\xi_1}\xi_2}\right).\\
		\end{aligned}
	\end{equation*}
	By definition one has
	\begin{equation*}
		\begin{aligned}
			&w_{q+1}^p+w_{q+1}^c=\frac{1}{N_\Lambda^2\lambda_{q+1}^2}\operatorname{curl}\operatorname{curl}\sum_{\xi\in\Lambda}a_\xi\phi_\xi\Psi_{\xi_1}\xi,\\
			&e_{q+1}^{i,p}+e_{q+1}^{i,c}=\frac{1}{N_\Lambda^2\lambda_{q+1}^2}\operatorname{curl}\operatorname{curl}\sum_{\xi\in\Lambda_i}a_\xi\phi_\xi\Psi_{\xi_1}\xi_2.
		\end{aligned}
	\end{equation*}
	Hence $\operatorname{div}(w_{q+1}^p+w_{q+1}^c)=\operatorname{div}(e_{q+1}^{i,p}+e_{q+1}^{i,c})=0 $.
	\subsection{Temporal Correctors}
	In addition to the incompressibility correctors, we introduce the temporal correctors, which are defined by
	\begin{align}
		w_{q+1}^t&=-\frac{1}{\mu}\sum_{\xi\in\Lambda}\mathbb{P}_H\mathbb{P}_{\neq 0}(a_\xi^2\phi_\xi^2\varphi_{\xi_1}^2)\xi,\label{delwt}\quad\text{and}\quad
		e_{q+1}^{i,t}=-\frac{1}{\mu}\sum_{\xi\in\Lambda_i}\mathbb{P}_H\mathbb{P}_{\neq 0}(a_\xi^2\phi_\xi^2\varphi_{\xi_1}^2)\xi_2.
	\end{align}
	Here $\mathbb{P}_{H}=\mathrm{Id}-\nabla(\Delta^{-1}\operatorname{div})$ denote the usual Helmholtz projector onto divergence free fields. 
	It is easy to verify that $w_{q+1}^t$ and $e_{q+1}^{i,t}$ are both divergence free and have zero mean.
	By definition we have
	\begin{align}
		\label{partial t wt}\partial_tw_{q+1}^t=&-\frac{1}{\mu}\sum_{\xi\in\Lambda}\mathbb{P}_{\neq 0}\partial_t(a_\xi^2\phi_\xi^2\varphi_{\xi_1}^2\xi)+\frac{1}{\mu}\sum_{\xi\in\Lambda}\nabla\left(\Delta^{-1}\operatorname{div}\partial_t(a_\xi^2\phi_\xi^2\varphi_{\xi_1}^2\xi)\right).\\
		\label{partial t et}\partial_te_{q+1}^{i,t}=&-\frac{1}{\mu}\sum_{\xi\in\Lambda_i}\mathbb{P}_{\neq 0}\partial_t(a_\xi^2\phi_\xi^2\varphi_{\xi_1}^2\xi_2)+\frac{1}{\mu}\sum_{\xi\in\Lambda_i}\nabla\left(\Delta^{-1}\operatorname{div}\partial_t(a_\xi^2\phi_\xi^2\varphi_{\xi_1}^2\xi_2)\right).
	\end{align}
	
	We define the perturbations:
	\begin{equation}\label{tpc}
		w_{q+1}:=w_{q+1}^p+w_{q+1}^c+w_{q+1}^t, ~~~e_{q+1}^{i}=e_{q+1}^{i,p}+e_{q+1}^{i,c}+e_{q+1}^{i,t},
	\end{equation}
	which are mean zero and divergence-free. 
	The new velocity field and deformation vectors are defined as
	\begin{equation}\label{def vq+1}
		v_{q+1}:=v_{q}+w_{q+1}, ~~~\quad\quad F_{q+1}^{i}=F_{q}^{i}+e_{q+1}^{i}.
	\end{equation}
	\subsection{Verification of Inductive Estimates for Perturbations}
	To estimate the $L^2$ norm of $w_{q+1}^p$ and $e_{q+1}^{i,p}$, we introduce the following $L^p$ Decorrelation Lemma from \cite{Luonon}:
	\begin{lem}\label{Lp Decorrelation}
		Let $f,g\in C^\infty(\mathbb{T}^3)$, and $g$ is $(\frac{\mathbb{T}}{\kappa})^3$ periodic, $\kappa\in\mathbb{N}$. Then for $1\leq p \leq \infty$,
		\begin{equation*}
			\|fg\|_{L^p}\lesssim \|f\|_{L^p}\|g\|_{L^p}+\kappa^{-\frac{1}{p}}\|f\|_{C^1}\|g\|_{L^p}.
		\end{equation*}
	\end{lem}
	{Now we prove the inductive estimates of the perturbations $w_{q+1}$ and $e_{q+1}^i$.}
	\begin{prop}\label{estimateper}
		The {perturbations $w_{q+1}$ and $e_{q+1}^i$}  obey the following bounds
		\begin{align}
			\label{L2 of perturbations}&\|w_{q+1}\|_{L^2}+\sum_{i=1}^3\|e^{i}_{q+1}\|_{L^2}\leq M\delta_{q+1}^\frac{1}{2},\\
			\label{Lp of perturbations}&\|w_{q+1}\|_{L^p}+\sum_{i=1}^3\|e_{q+1}^{i}\|_{L^p}\lesssim\ell^{-3}(\sigma r)^{\frac{1}{p}-\frac{1}{2}},\\
			\label{W1p}&\|w_{q+1}\|_{W^{1,p}}+\sum_{i=1}^3\|e^i_{q+1}\|_{W^{1,p}}\lesssim\ell^{-9}\lambda_{q+1}(\sigma r)^{\frac{1}{p}-\frac{1}{2}},\\
			\label{C1 of perturbations}&\|w_{q+1}\|_{C^1_{t,x}}+\sum_{i=1}^3\|e^{i}_{q+1}\|_{C^1_{t,x}}\leq\frac{1}{10}\lambda_{q+1}^{4},
		\end{align}
		for $p\in[1,+\infty]$.  Here the constant $M$ depends only on $\varepsilon_v$ and $\varepsilon_F$.
	\end{prop}
	\begin{proof}
		Applying Lemma \ref{Lp Decorrelation} with $f=a_\xi$, $g=\phi_\xi\varphi_{\xi_1}$, $\kappa=\lambda_{q+1} \sigma$, we get the following estimate
		\begin{equation*}
			\begin{aligned}
				\|w_{q+1}^p\|_{L^2}+\sum_{i=1}^3\|e_{q+1}^{i,p}\|_{L^2}&\lesssim \sum_{\xi \in \Lambda}\|a_\xi\|_{L^2}\|\phi_\xi\varphi_{\xi_1}\|_{L^2}+(\lambda_{q+1} \sigma)^{-\frac{1}{2}}\|a_\xi\|_{C^1}\|\phi_\xi\varphi_{\xi_1}\|_{L^2}\\
				&\lesssim\delta_{q+1}^\frac{1}{2}+\ell^{-3}(\lambda_{q+1}\sigma)^{-\frac{1}{2}}.
			\end{aligned}
		\end{equation*}
		In view of Lemma \ref{philemma} and Lemma \ref{lemma estimate of a},  we obtain for any $p\in[1,\infty]$
		\begin{equation}\label{Lpprinciple of w}
			\begin{aligned}
				\|w_{q+1}^p\|_{L^p}+\|e_{q+1}^{i,p}\|_{L^p}&\lesssim\|a_\xi\|_{C^1}\|\phi_\xi\Psi_{\xi_1}\|_{L^p}\lesssim\ell^{-3}(\sigma r)^{\frac{1}{p}-\frac{1}{2}},
			\end{aligned}
		\end{equation}
		and
		\begin{equation*}
			\begin{aligned}
				\|w_{q+1}^c\|_{L^p}
				&\lesssim\frac{1}{\lambda_{q+1}^2}\left(\|\nabla^2a_\xi\phi_\xi\Psi_{\xi_1}\xi\|_{L^p}+\|\nabla a_\xi\times\operatorname{curl}(\phi_\xi\Psi_{\xi_1}\xi)\|_{L^p}+\| a_\xi\nabla\phi_\xi\times\operatorname{curl}\Psi_{\xi_1}\xi\|_{L^p}\right)\\
				&\lesssim\frac{1}{\lambda_{q+1}^2}(\|a_\xi\|_{C^2}\|\phi_\xi\Psi_{\xi_1}\|_{L^{p}}+\|a_\xi\|_{C^1}\|\phi_\xi\Psi_{\xi_1}\|_{W^{1,p}}+\|a_\xi\|_{C^0}\|\phi_\xi\|_{W^{1,p}}\|\Psi_{\xi_1}\|_{W^{1,p}})\\
				&\lesssim \lambda_{q+1}^{-2}\left(\ell^{-6}(\sigma r)^{\frac{1}{p}-\frac{1}{2}}+\ell^{-3}\lambda_{q+1}(\sigma r)^{\frac{1}{p}-\frac{1}{2}}+\ell^{-1}\lambda_{q+1}\sigma r^{\frac{1}{p}-\frac{3}{2}}\lambda_{q+1}\sigma^{\frac{1}{p}-\frac{1}{2}}\right)\\
				&\lesssim \lambda_{q+1}^{-2}\ell^{-6}(\sigma r)^{\frac{1}{p}-\frac{1}{2}}(1+\lambda_{q+1}+\lambda_{q+1}^2\sigma r^{-1})\lesssim\ell^{-6}\sigma^{\frac{1}{p}+\frac{1}{2}}{r}^{\frac{1}{p}-\frac{3}{2}}.
			\end{aligned}
		\end{equation*}
		Analogously, for $e_{q+1}^{i,c}$ we also have
		\begin{equation}\label{Lpcorrector of e}
			\|e_{q+1}^{i,c}\|_{L^p}\lesssim\ell^{-6}\sigma^{\frac{1}{p}+\frac{1}{2}}{r}^{\frac{1}{p}-\frac{3}{2}}.
		\end{equation}
		By the definition \eqref{delwt} of $w_{q+1}^t$ we get
		\begin{equation*}
			\|w_{q+1}^t\|_{L^p}\lesssim\frac{1}{\mu}\|a_\xi\|_{C^1}^2\|\phi_\xi\varphi_{\xi_1}\|^2_{L^{2p}}\lesssim\ell^{-6}{\mu}^{-1}(\sigma r)^{\frac{1}{p}-1}.
		\end{equation*}
		For $e_{q+1}^{i,t}$ we also arrive at
		\begin{equation}\label{Lptemporal of e}
			\begin{aligned}
				\|e_{q+1}^{i,t}\|_{L^p}\lesssim\frac{1}{\mu}\|a_\xi\|^2_{C^1}\|\phi_\xi\|_{L^{2p}}^2\|\varphi_{\xi_1}\|_{L^{2p}}^2\lesssim\ell^{-6}{\mu}^{-1}(\sigma r)^{\frac{1}{p}-1}.
			\end{aligned}
		\end{equation}
		The above estimates yield
		\begin{equation*}
			\begin{aligned}
				\|w_{q+1}\|_{L^2}+\sum_{i=1}^3\|e_{q+1}^{i}\|_{L^2}&\lesssim\delta_{q+1}^\frac{1}{2}+\ell^{-3}(\lambda_{q+1}\sigma)^{-\frac{1}{2}}+\ell^{-6}\frac{\sigma}{r}+\ell^{-6}{\mu}^{-1}(\sigma r)^{-\frac{1}{2}}\lesssim\delta_{q+1}^\frac{1}{2},\\
				\|w_{q+1}\|_{L^p}+\sum_{i=1}^3\|e_{q+1}^{i}\|_{L^p}&\lesssim\ell^{-3}(\sigma r)^{\frac{1}{p}-\frac{1}{2}}+\ell^{-6}\sigma^{\frac{1}{p}+\frac{1}{2}}{r}^{\frac{1}{p}-\frac{3}{2}}+\ell^{-6}{\mu}^{-1}(\sigma r)^{\frac{1}{p}-1}\lesssim\ell^{-3}(\sigma r)^{\frac{1}{p}-\frac{1}{2}},
			\end{aligned}
		\end{equation*}
		where we used Definition \eqref{para def} and parameter inequality \eqref{para ineq}. The implicit constants in the above estimates only depend on $\varepsilon_v$ and $\varepsilon_F$. Hence we can choose a large $M$ such that \eqref{L2 of perturbations} holds.
		
		Now we consider \eqref{W1p}. By Lemma \ref{philemma} and Lemma \ref{lemma estimate of a} we have
		\begin{equation}\label{W1PP}
			\begin{aligned}
				\left\|w_{q+1}^p\right\|_{W^{1,p}}+\sum_{i=1}^3\left\|e_{q+1}^{i,p}\right\|_{W^{1,p}}\lesssim\sum_{\xi\in\Lambda}\|a_\xi\|_{C^1}\|\phi_\xi\varphi_{\xi_1}\|_{W^{1,p}}\lesssim\ell^{-3}\lambda_{q+1}\sigma^{\frac{1}{p}-\frac{1}{2}}r^{\frac{1}{p}-\frac{1}{2}}.
			\end{aligned}
		\end{equation}
		Moreover, we obtain from \eqref{phiparaLp} and \eqref{phiperpLp} that $\|\phi_\xi\|_{W^{k,p}}\lesssim (\lambda_{q+1}\sigma)^kr^{\frac{1}{p}-\frac{1}{2}-k}$ and $\|\Psi_{\xi_1}\|_{W^{k,p}}\lesssim \lambda_{q+1}^k\sigma^{\frac{1}{p}-\frac{1}{2}}$, $k=0,1,2,3$.
		Then for $w_{q+1}^c$ and $e_{q+1}^{i,c}$ we have
		\begin{equation*}
			\begin{aligned}
				&\left\|w_{q+1}^c\right\|_{W^{1,p}}+\sum_{i=1}^3\left\|e_{q+1}^{i,c}\right\|_{W^{1,p}}\\
				&\lesssim\sum_{\xi\in\Lambda}\lambda_{q+1}^{-2}\|a_\xi\|_{C^3}\left(\|\phi_\xi\|_{W^{1,p}}\|\Psi_{\xi_1}\|_{W^{2,p}} +\|\phi_\xi\|_{W^{2,p}}\|\Psi_{\xi_1}\|_{W^{1,p}}+\|\phi_\xi\|_{W^{3,p}}\|\Psi_{\xi_1}\|_{L^{p}}\right)\\ &\lesssim\lambda_{q+1}^{-2}\ell^{-9}\left((\lambda_{q+1}\sigma) r^{\frac{1}{p}-\frac{3}{2}}\lambda_{q+1}^2+(\lambda_{q+1}\sigma)^2 r^{\frac{1}{p}-\frac{5}{2}}\lambda_{q+1}+(\lambda_{q+1}\sigma)^3 r^{\frac{1}{p}-\frac{7}{2}}\right)\sigma^{\frac{1}{p}-\frac{1}{2}}\\
				&\lesssim\lambda_{q+1}\ell^{-9} \sigma^{\frac{1}{p}+\frac{1}{2}}r^{\frac{1}{p}-\frac{3}{2}},
			\end{aligned}
		\end{equation*}
		where we used the fact $\sigma\ll r$. Similarly, in order to estimate $w_{q+1}^t$ and $e_{q+1}^{i,t}$, using \eqref{phiproduct} and Lemma \ref{lemma estimate of a} one obtains
		\begin{equation}\label{W1PT}
			\|w_{q+1}^t\|_{W^{1,p}}+\|e_{q+1}^{i,t}\|_{W^{1,p}}\lesssim\sum_{\xi\in\Lambda}\frac{1}{\mu}\|a_\xi\|_{C^1}^2\|\phi_\xi^2\varphi_{\xi_1}^2\|_{W^{1,p}}\lesssim\ell^{-6}\mu^{-1}\lambda_{q+1}\sigma^{\frac{1}{p}-1}r^{\frac{1}{p}-1}.
		\end{equation}
		We conclude from \eqref{W1PP}-\eqref{W1PT} that
		\begin{equation*}
			\begin{aligned}
				\|w_{q+1}\|_{W^{1,p}}+\|e_{q+1}^{i}\|_{W^{1,p}}&\lesssim\ell^{-3}\lambda_{q+1}\sigma^{\frac{1}{p}-\frac{1}{2}}r^{\frac{1}{p}-\frac{1}{2}}+\ell^{-9}\lambda_{q+1} \sigma^{\frac{1}{p}+\frac{1}{2}}r^{\frac{1}{p}-\frac{3}{2}}+\ell^{-6}\mu^{-1}\lambda_{q+1}\sigma^{\frac{1}{p}-1}r^{\frac{1}{p}-1}\\
				&\lesssim\ell^{-3}\lambda_{q+1}\sigma^{\frac{1}{p}-\frac{1}{2}}r^{\frac{1}{p}-\frac{1}{2}},
			\end{aligned}
		\end{equation*}
		which yields \eqref{W1p}.
		
		Finally, we estimate the $C^1$ norm of $w_{q+1}$ and $e_{q+1}^{i}$. Applying Lemma \ref{philemma} we obtain
		\begin{equation*}
			\begin{aligned}
				\|w_{q+1}^p+w_{q+1}^c\|_{C^1_{t,x}}+\sum_{i=1}^3\|e_{q+1}^{i,p}+e_{q+1}^{i,c}\|_{C^1_{t,x}}
				\lesssim&\sum_{\xi\in\Lambda}\lambda_{q+1}^{-2}\|\operatorname{curl}\operatorname{curl}(a_\xi\phi_\xi\Psi_{\xi_1})\|_{C^1_{t,x}}\\
				\lesssim&\lambda_{q+1}^{-2}\ell^{-9}\lambda_{q+1}^3
				r^{-\frac{1}{2}}\sigma^{-\frac{1}{2}}\ll\lambda_{q+1}^4,\\
			\end{aligned}
		\end{equation*}
		and
		\begin{align*}
			\|w_{q+1}^t\|_{C^1_{t,x}}+\sum_{i=1}^3\|e_{q+1}^{i,t}\|_{C^1_{t,x}}\lesssim\mu^{-1}\sum_{\xi\in\Lambda}\|a_\xi^2\|_{C^1_{t,x}}\|\phi_\xi^2\|_{C^1_{t,x}}\|\varphi_{\xi_1}^2\|_{C^1_{t,x}}
			\lesssim\mu^{-1}\ell^{-6}\lambda_{q+1}^2\sigma^{-1}r^{-1}\ll\lambda_{q+1}^{4}.
		\end{align*}
		This completes the proof of \eqref{C1 of perturbations}.  
	\end{proof}
	
	Applying standard mollification estimates we have
	\begin{gather}
		\label{vq vl L2}\|v_q-v_\ell\|_{L^2}\lesssim \|v_q-v_\ell\|_{C^0}\lesssim\ell\|v_q\|_{C^1}\lesssim\ell\lambda_q^4\ll\delta_{q+1}^\frac{1}{2},\\
		\label{Eq El L2}\|F^i_q-F^i_\ell\|_{L^2}\lesssim \|F^i_q-F^i_\ell\|_{C^0}\lesssim\ell\|F^i_q\|_{C^1}\lesssim\ell\lambda_q^4\ll\delta_{q+1}^\frac{1}{2},
	\end{gather}
	where we used the fact that $\beta b^2\leq \frac{1}{16}$.
	Combining \eqref{L2 of perturbations}, \eqref{vq vl L2}, and \eqref{Eq El L2} we obtain
	\begin{gather*}
		\|v_{q+1}-v_{q}\|_{L^2}\lesssim \|w_{q+1}\|_{L^2}+\|v_q-v_\ell\|_{L^2}\lesssim\delta_{q+1}^\frac{1}{2},\\
		\|F^i_{q+1}-F^i_{q}\|_{L^2}\lesssim \|e^i_{q+1}\|_{L^2}+\|F^i_q-F^i_\ell\|_{L^2}\lesssim\delta_{q+1}^\frac{1}{2}.
	\end{gather*}
	Hence the solutions we construct satisfy the inductive estimate \eqref{iteration L2}. Similarly, one can check that 
	\begin{equation*}
		\|v_q-v_\ell\|_{L^1}+\|F^i_{q+1}-F^i_{q}\|_{L^1}\lesssim \ell\lambda_q^4\ll\delta_{q+2}.
	\end{equation*}
	Combining this with \eqref{Lp of perturbations} to obtain 
	\begin{align*}
		\|v_{q+1}-v_{q}\|_{L^1}+\|F^i_{q+1}-F^i_{q}\|_{L^1}\lesssim\|w_{q+1}\|_{L^1}+\|e^i_{q+1}\|_{L^1}+\|v_q-v_\ell\|_{L^1}+\|F^i_{q+1}-F^i_{q}\|_{L^1}\ll \delta_{q+2},
	\end{align*}
	which yields the inductive estimate \eqref{iteration L1}. We then check the $C^1$ estimate for $v_{q+1}$ and $F^i_{q+1}$. Using  \eqref{vl CN} and \eqref{C1 of perturbations} to obtain that 
	\begin{gather*}
		\|v_{q+1}\|_{C^1_{t,x}}\leq\|v_\ell\|_{C^1_{t,x}}+ \|w_{q+1}\|_{C^1_{t,x}}\leq \frac{1}{5}\lambda_{q+1}^4,   \\
		\|F^i_{q+1}\|_{C^1_{t,x}}\leq\|F^i_\ell\|_{C^1_{t,x}}+\|e^{i}_{q+1}\|_{C^1_{t,x}}\leq \frac{1}{5}\lambda_{q+1}^4.
	\end{gather*}
	This proves the inductive estimates we have claimed in \eqref{C1norm}. Finally, we estimate  $\operatorname{supp}_t(v_{q+1},F_{q+1}^1,F_{q+1}^2,F_{q+1}^3)$. By definition, it is easy to check that the amplitude functions satisfy
	\begin{align}\label{suppa}
		\bigcup_{\xi\in\Lambda}\operatorname{supp}_ta_\xi\subseteq \operatorname{supp}_t\Phi_q\subset O_\ell\left(\operatorname{supp}_t (R^v_\ell,R^1_\ell,R^2_\ell,R^3_\ell)\right)\subseteq O_{2\ell}\left(\operatorname{supp}_t (R^v_q,R^1_q,R^2_q,R^3_q)\right).
	\end{align}
	Recall that $\ell=\lambda_q^{-20}$ and $\delta_{q+2}=\lambda_1^{3\beta}\lambda_q^{-2\beta b^2}$, by \eqref{para ineq} one gets $10\ell\leq \delta_{q+2}$. We obtain  that 
	\begin{align*}
		\operatorname{supp}_t(w_{q+1},e_{q+1}^1,e_{q+1}^2,e_{q+1}^3)\subseteq  \bigcup_{\xi\in\Lambda}\operatorname{supp}_ta_\xi \subseteq O_{\delta_{q+2}}\left(\operatorname{supp}_t (R^v_q,R^1_q,R^2_q,R^3_q)\right).
	\end{align*}
	Combining this with \eqref{suppmol} yields that 
	\begin{equation}
		\begin{aligned}\label{incresupp}
			\operatorname{supp}_t(v_{q+1},F_{q+1}^1,F_{q+1}^2,F_{q+1}^3)&\subseteq    \operatorname{supp}_t(w_{q+1},e_{q+1}^1,e_{q+1}^2,e_{q+1}^3)\cup\operatorname{supp}_t(v_{\ell},F_{\ell}^1,F_{\ell}^2,F_{\ell}^3)\\
			&\subseteq  \bigcup_{i=1}^3O_{\delta_{q+2}}\left(\operatorname{supp}_t(v_{q},R^v_{q},F^{i}_{q},R^i_{q})\right).
		\end{aligned}
	\end{equation}
	
	\section{The Estimate of Stresses}\label{eststress}
	
	\subsection{Decomposition of the Stresses}
	Our goal is to show that $R_{q+1}^v$ and $R_{q+1}^i$ satisfy the inductive estimates \eqref{STRESS L1}, \eqref{STRESS C1} and \eqref{itersupp}. Recall that $v_{q+1}=v_\ell+w_{q+1}$, and $F_{q+1}^i=F_\ell^i+e_{q+1}^i$.  Using \eqref{mollified elastic}, \eqref{tpc} and \eqref{def vq+1} we obtain
	\begin{equation}
		\begin{aligned}
			\operatorname{div}{R}_{q+1}^{i}
			=&\underbrace{\partial_{t} e^i_{q+1}+\operatorname{div}\left( e^i_{q+1}\otimes v_\ell+F_\ell^i \otimes w_{q+1}- w_{q+1}\otimes F^i_\ell- v_\ell\otimes e^i_{q+1} \right)}_{\operatorname{div} {R}_{l i n}^{i}} \\
			&+\operatorname{div}\bigg(e_{q+1}^i \otimes (w_{q+1}^{c}+w_{q+1}^{t}) - (w_{q+1}^{c}+w_{q+1}^{t})\otimes e_{q+1}^i\notag\\
			&\underbrace{~~~~~+(e_{q+1}^{i,c}+e_{q+1}^{i,t})\otimes w_{q+1}^{p} - w_{q+1}^{p}\otimes (e_{q+1}^{i,c}+e_{q+1}^{i,t})\bigg)
			}_{\operatorname {div } R_{ {corr}}^{i}}\\
			&\underbrace{+\operatorname{div}\left( e_{q+1}^{i,p}\otimes w_{q+1}^{p} - w_{q+1}^{p}\otimes e_{q+1}^{i,p} +{R}_\ell^i\right)+\partial_te_{q+1}^{i,t}}_{\operatorname{div} {R}_{o s c}^{i}} +\operatorname{div}R^i_{comm}
		\end{aligned}
	\end{equation}
	and
	\begin{equation*}
		\begin{aligned}
			&\operatorname{div} {R}_{q+1}^{v}-\nabla p_{q+1} \\
			&\quad=\underbrace{(-\Delta)^\theta w_{q+1}+\partial_{t}(w_{q+1}^p+w_{q+1}^c)+\operatorname{div}\left(v_\ell \otimes w_{q+1}+w_{q+1} \otimes v_\ell-F^i_\ell \otimes e^i_{q+1}-e^i_{q+1} \otimes F^i_\ell\right)}_{\operatorname{div} R_{\operatorname{lin}}^{v}+\nabla p_{l i n}} \\
			&\quad\quad+\underbrace{\operatorname{div}\bigg(w_{q+1} \otimes (w_{q+1}^{c}+w_{q+1}^{t})+(w_{q+1}^{c}+w_{q+1}^{t}) \otimes w_{q+1}^{p}-e^i_{q+1} \otimes (e_{q+1}^{i,c}+e_{q+1}^{i,t})-(e_{q+1}^{i,c}+e_{q+1}^{i,t}) \otimes e_{q+1}^{i,p}\bigg)}_{\operatorname {div } {R}_{c o r r}^{v}+\nabla{p_{c o r r}}} \\
			&\quad\quad+\underbrace{\operatorname{div}\left(w_{q+1}^{p} \otimes w_{q+1}^{p}-e_{q+1}^{i,p} \otimes e_{q+1}^{i,p}+{R}_\ell^{v}\right)+\partial_tw_{q+1}^{t}}_{\operatorname{div} {R}_{o s c}^{v}+\nabla p_{o s c}}+\operatorname{div}R^v_{comm}-\nabla p_\ell.
		\end{aligned}
	\end{equation*}
	\subsection{Linear and Corrector Errors }
	The  symmetric and skew-symmetric inverse divergence operators allow us to define different parts of the Reynolds stresses as follows:
	\begin{equation}\label{Rlico}
		\begin{aligned}
			&{R}_{lin}^{i}=\mathcal{R}^{F}\left(\partial_{t}(e^{i,p}_{q+1}+e^{i,c}_{q+1})\right)+\left(e^i_{q+1}\otimes v_\ell+F_\ell^i \otimes w_{q+1}- w_{q+1}\otimes F^i_\ell- v_\ell\otimes e^i_{q+1}\right), \\
			&R_{corr}^{i}=e_{q+1}^i\otimes(w_{q+1}^{c}+w_{q+1}^{t})-(w_{q+1}^{c}+w_{q+1}^{t})\otimes e_{q+1}^i+(e_{q+1}^{i,c}+e_{q+1}^{i,t})\otimes w_{q+1}^{p}-w_{q+1}^{p}\otimes(e_{q+1}^{i,c}+e_{q+1}^{i,t}), \\
			&{R}_{lin}^{v}=\mathcal{R}(-\Delta)^\theta w_{q+1}+\mathcal{R}\partial_{t}(w_{q+1}^{p}+w_{q+1}^{c})+v_\ell\hat{\otimes}w_{q+1}+w_{q+1}\hat{\otimes}v_\ell-F^i_\ell\hat{\otimes}e^i_{q+1}-e^i_{q+1}\hat{\otimes} F^i_\ell,\\
			&{R}_{corr}^{v}=w_{q+1}\hat{\otimes}(w_{q+1}^{c}+w_{q+1}^{t})+(w_{q+1}^{c}+w_{q+1}^{t}) \hat{\otimes} w_{q+1}^{p}-e^i_{q+1} \hat{\otimes}(e_{q+1}^{i,c}+e_{q+1}^{i,t})-(e_{q+1}^{i,c}+e_{q+1}^{i,t}) \hat{\otimes} e_{q+1}^{i,p}.
		\end{aligned}
	\end{equation}
	We first estimate the linear errors. Recall that
	\begin{equation*}
		\begin{aligned}
			&w_{q+1}^{p}+w_{q+1}^{c}=\frac{1}{N_\Lambda^2\lambda_{q+1}^2}\operatorname{curl}\operatorname{curl}\sum_{\xi\in\Lambda}a_\xi\phi_\xi\Psi_{\xi_1}\xi,\\
			&e^{i,p}_{q+1}+e^{i,c}_{q+1}=\frac{1}{N_\Lambda^2\lambda_{q+1}^2}\operatorname{curl}\operatorname{curl}\sum_{\xi\in\Lambda_i}a_\xi\phi_\xi\Psi_{\xi_1}\xi_2.
		\end{aligned}
	\end{equation*}
	For $p\in(1,2)$, by \eqref{est antidiv}, \eqref{phiparaLp}, \eqref{phiperpLp} and \eqref{aCN} we have
	\begin{equation*}
		\begin{aligned}
			\|\mathcal{R}\left(\partial_{t}(w_{q+1}^{p}+w_{q+1}^{c})\right)\|_{L^p}\lesssim& \lambda_{q+1}^{-2} \sum_{\xi\in\Lambda}\|\mathcal{R}\operatorname{curl}\operatorname{curl}\partial_t \left(a_\xi\phi_\xi\Psi_{\xi_1}\xi\right)\|_{L^p}\\
			\lesssim&\lambda_{q+1}^{-2}\left(\|a_\xi\|_{C^2} \|\partial_t\nabla\phi_\xi\|_{L^p}\|\Psi_{\xi_1}\|_{L^p}+\|a_\xi\|_{C^2} \|\partial_t\phi_\xi\|_{L^p}\|\nabla\Psi_{\xi_1}\|_{L^p}\right)\\
			\lesssim&\lambda_{q+1}^{-2}\left(\ell^{-6}\mu\left(\frac{\lambda_{q+1}\sigma}{r}\right)^2 \sigma^{\frac{1}{p}-\frac{1}{2}}r^{\frac{1}{p}-\frac{1}{2}}+\ell^{-6}\mu \frac{\lambda_{q+1}\sigma}{r}\sigma^{\frac{1}{p}-\frac{1}{2}}\lambda_{q+1}r^{\frac{1}{p}-\frac{1}{2}}\right)\\
			\lesssim&\ell^{-6}\mu\sigma^{\frac{1}{p}+\frac{1}{2}}r^{\frac{1}{p}-\frac{3}{2}}.
		\end{aligned}
	\end{equation*}
	Similar result holds for $\mathcal{R}^{F}\left(\partial_{t} (e^{i,p}_{q+1}+e^{i,c}_{q+1})\right)$ in $R_{lin}^i$. Next we estimate the high-low interaction terms in linear errors. In view of \eqref{vl CN}, \eqref{Lpprinciple of w}, \eqref{Lpcorrector of e} and \eqref{Lptemporal of e}, one has
	\begin{equation*}
		\begin{aligned}
			\|v_\ell\otimes e^i_{q+1} \|_{L^1}
			\leq&\left\|v_{\ell} \otimes e_{q+1}^{i,p}\right\|_{L^{1}}+\left\|v_{\ell} \otimes e_{q+1}^{i,c}\right\|_{L^{1}}+\left\|v_{\ell} \otimes e_{q+1}^{i,t}\right\|_{L^{1}} \\ \leq&\left\|v_{\ell}\right\|_{C^{0}}\left\|e_{q+1}^{i,p}\right\|_{L^{1}}+\left\|v_{\ell}\right\|_{C^{0}}\left\|e_{q+1}^{i,c}\right\|_{L^1}+\left\|v_{\ell}\right\|_{C^{0}}\left\|e_{q+1}^{i,t}\right\|_{L^{1}} \\ \lesssim& \lambda_q^4\ell^{-3}(\sigma r)^\frac{1}{2}+\lambda_q^4\ell^{-6}\sigma^{\frac{3}{2}}r^{-\frac{1}{2}}+\lambda_q^4\ell^{-6}\mu^{-1}\\
			\lesssim&\lambda_q^4\ell^{-6} (\sigma r)^{\frac{1}{2}}.
		\end{aligned}
	\end{equation*}
	Other high-low interaction terms can be handled similarly. For the dissipation term, combining \eqref{Lp of perturbations} and \eqref{W1p} we obtain
	\begin{equation*}
		\|\mathcal{R}(-\Delta)^\theta w_{q+1}\|_{L^p}\lesssim\|w_{q+1}\|^{1-\theta_*}_{L^p}\|w_{q+1}\|_{W^{1,p}}^{\theta_*}\lesssim\ell^{-9}\lambda_{q+1}^{\theta_*}(\sigma r)^{\frac{1}{p}-\frac{1}{2}},
	\end{equation*}
	where $\theta_*$ is defined by \eqref{theta}.
	Thus we obtain
	\begin{equation}\label{estimate linv}
		\begin{aligned}
			\|{R}_{l i n}^{v}\|_{L^1}&\lesssim\|\mathcal{R}(-\Delta)^\theta w_{q+1}\|_{L^p}+\|\mathcal{R}\left(\partial_{t} (w_{q+1}^{p}+w_{q+1}^{c})\right)\|_{L^p}\\
			&\quad\quad+\|v_\ell\hat{\otimes}w_{q+1}+w_{q+1}\hat{\otimes}v_\ell-F^i_\ell\hat{\otimes}e^i_{q+1}-e^i_{q+1}\hat{\otimes} F^i_\ell\|_{L^1}\\
			&\lesssim\ell^{-9}\lambda_{q+1}^{\theta_*}(\sigma r)^{\frac{1}{p}-\frac{1}{2}}       +\ell^{-6}\mu\sigma^{\frac{1}{p}+\frac{1}{2}}r^{\frac{1}{p}-\frac{3}{2}}+\lambda_q^4\ell^{-6} (\sigma r)^{\frac{1}{2}}\\
			&\lesssim\ell^{-9}(\sigma r)^{\frac{1}{p}-\frac{1}{2}}(\lambda_{q+1}^{\theta_*}+\mu\sigma r^{-1}).
		\end{aligned}
	\end{equation}
	Analogously we get
	\begin{equation*}
		\begin{aligned}
			\|{R}_{l i n}^{i}\|_{L^1}&\lesssim\|\mathcal{R}^F\partial_{t}\left( e^{i,p}_{q+1}+e^{i,c}_{q+1}\right)\|_{L^p}+\|e^i_{q+1}\otimes v_\ell+F_\ell^i \otimes w_{q+1}- w_{q+1}\otimes F^i_\ell- v_\ell\otimes e^i_{q+1}\|_{L^1}\\
			&\lesssim\ell^{-6}\mu\sigma^{\frac{1}{p}+\frac{1}{2}}r^{\frac{1}{p}-\frac{3}{2}}+\lambda_q^4\ell^{-6} (\sigma r)^{\frac{1}{2}}\\
			&\lesssim\ell^{-6}\mu\sigma^{\frac{1}{p}+\frac{1}{2}}r^{\frac{1}{p}-\frac{3}{2}}.
		\end{aligned}
	\end{equation*}
	Next we estimate the corrector errors. Due to the smallness of the corrector terms, we apply \eqref{L2 of perturbations}, \eqref{Lpcorrector of e}, \eqref{Lptemporal of e}, and obtain
	\begin{equation}\label{estimate corri}
		\begin{aligned}
			\|R_{corr}^i\|_{L^1}&\lesssim \|w_{q+1}^c+w_{q+1}^t\|_{L^2}\|e_{q+1}^i\|_{L^2}+\|w_{q+1}\|_{L^2}\|e_{q+1}^{i,c}+e_{q+1}^{i,t
			}\|_{L^2}\\
			&\lesssim (\ell^{-6}\sigma r^{-1}+\ell^{-6}\mu^{-1}(\sigma r)^{-\frac{1}{2}})\delta_{q+1}^{\frac{1}{2}}\\
			&\lesssim\ell^{-6}\mu^{-1}(\sigma r)^{-\frac{1}{2}}\delta_{q+1}^{\frac{1}{2}}.
		\end{aligned}
	\end{equation}
	And $R_{corr}^v$ obeys the same estimate.
	\subsection{Oscillation Error}\label{secosc}
	By \eqref{cancelRi} we have
	\begin{equation*}
		\begin{aligned}
			&\operatorname{div}\left(e_{q+1}^{i,p}\otimes w_{q+1}^{p}-w_{q+1}^{p}\otimes e_{q+1}^{i,p}+R^i_\ell \right)\\
			&=\operatorname{div}\left(\sum_{\xi\in \Lambda_{i}}a_{\xi}^{2}\phi_{\xi}^{2}\varphi_{\xi_1}^{2}(\xi_{2}\otimes\xi-\xi\otimes\xi_{2})+R^i_\ell\right)+\operatorname{div}\left(\sum_{\substack{\xi\in \Lambda, \xi' \in \Lambda_{i}\\ \xi \neq \xi^{'} }} a_{\xi} a_{\xi^{'}} \phi_{\xi} \varphi_{\xi_1} \phi_{\xi'}\varphi_{\xi_1'}(\xi'_{2} \otimes \xi-\xi \otimes \xi'_{2})\right)\\
			&=\operatorname{div}\left(\sum_{\xi\in \Lambda_{i}}a_{\xi}^{2}\mathbb{P}_{\geq \lambda_{q+1}\sigma}(\phi_\xi^2\varphi_{\xi_1}^2)(\xi_{2}\otimes\xi-\xi\otimes\xi_{2})\right)+\operatorname{div}\left(\sum_{\substack{\xi\in \Lambda, \xi' \in \Lambda_{i}\\ \xi \neq \xi^{'} }} a_{\xi} a_{\xi^{'}} \phi_{\xi} \varphi_{\xi_1} \phi_{\xi'}\varphi_{\xi_1'}(\xi'_{2} \otimes \xi-\xi \otimes \xi'_{2})\right)\\
			&:=\operatorname{div}(E_{i,1}+E_{i,2}),
		\end{aligned}
	\end{equation*}
	where we also used that $\phi_\xi$ and $\varphi_{\xi_1}$ are $(\frac{\mathbb{T}}{\lambda_{q+1}\sigma})$ periodic, hence $\mathbb{P}_{\neq 0}(\phi_\xi^2\varphi_{\xi_1}^2)=\mathbb{P}_{\geq \lambda_{q+1}\sigma}(\phi_\xi^2\varphi_{\xi_1}^2)$.
	
	The oscillation stress is given by
	\begin{equation}\label{Riosc}
		\begin{aligned}
			R_{osc}^i=\mathcal{R}^F(\operatorname{div}E_{i,1}+\partial_te_{q+1}^{i,t})+E_{i,2}.
		\end{aligned}
	\end{equation}
	It is easy to check that $\operatorname{div}(\operatorname{div}E_{i,1})=0$, hence the term $\mathcal{R}^F(\operatorname{div}E_{i,1}+\partial_te_{q+1}^{i,t})$ is well-defined.
	In view of \eqref{phippproduct}, Lemma \ref{lemma estimate of a}, and  Lemma \ref{Lp Decorrelation}, the term $E_{i,2}$ can be easily estimated as
	\begin{equation}\label{osc1}
		\begin{aligned}
			\|E_{i,2}\|_{L^1}&\lesssim\sum_{\substack{\xi\in \Lambda, \xi' \in \Lambda_{i}\\ \xi \neq \xi^{'} }}\|a_{\xi} a_{\xi^{'}} \phi_{\xi} \varphi_{\xi_1} \phi_{\xi'}\varphi_{\xi_1'}(\xi'_{2} \otimes \xi-\xi \otimes \xi'_{2})\|_{L^1}\\
			&\lesssim\sum_{\substack{\xi\in \Lambda, \xi' \in \Lambda_{i}\\ \xi \neq \xi^{'} }}\|a_\xi a_{\xi'}\|_{L^1}\|\phi_{\xi} \varphi_{\xi_1} \phi_{\xi'}\varphi_{\xi_1'}\|_{L^1}\lesssim\delta_{q+1}\frac{\sigma}{r}.
		\end{aligned}
	\end{equation}
	Then we focus on the term $E_{i,1}$. By the definition of $\phi_\xi$ and $\varphi_{\xi_1}$, it is easy to check that  $\xi_2\cdot \nabla(\phi_\xi^2 \varphi^2_{\xi_1})=0$ and $\xi\cdot \nabla \varphi_{\xi_1}^2=0$, then
	one has
	\begin{align*}
		\operatorname{div}\left((\phi_\xi^2\varphi_{\xi_1}^2)(\xi_{2}\otimes\xi-\xi\otimes\xi_{2})\right)=\xi\cdot \nabla(\phi_\xi^2 \varphi^2_{\xi_1})\xi_2-\xi_2\cdot \nabla(\phi_\xi^2 \varphi^2_{\xi_1})\xi=(\xi\cdot \nabla\phi_\xi^2 )\varphi^2_{\xi_1}\xi_2.
	\end{align*}
	Using \eqref{phi t} we have
	\begin{equation}\label{Ei1}
		\begin{aligned}
			\operatorname{div}E_{i,1}&=\sum_{\xi\in \Lambda_{i}}\mathbb{P}_{\neq 0}\left(\nabla a_\xi^2\mathbb{P}_{ \geq \lambda_{q+1}\sigma}(\phi_{\xi}^{2}\varphi_{\xi_1}^{2})(\xi_{2}\otimes\xi-\xi\otimes\xi_{2})\right)+\sum_{\xi\in \Lambda_{i}}\mathbb{P}_{\neq 0}\left(a_\xi^2(\xi\cdot\nabla\phi_{\xi}^{2})\varphi_{\xi_1}^{2}\xi_2\right)\\
			&=\sum_{\xi\in \Lambda_{i}}\mathbb{P}_{\neq 0}\left(\nabla a_\xi^2\mathbb{P}_{ \geq \lambda_{q+1}\sigma}(\phi_{\xi}^{2}\varphi_{\xi_1}^{2})(\xi_{2}\otimes\xi-\xi\otimes\xi_{2})\right)+\frac{1}{\mu}\sum_{\xi\in \Lambda_{i}}\mathbb{P}_{\neq 0}\left(a_\xi^2\partial_t\phi_{\xi}^{2}\varphi_{\xi_1}^{2}\xi_2\right).
		\end{aligned}
	\end{equation}
	
	Notice that the amplitude term $ a_\xi^2$ oscillates at a frequency much lower than $\lambda_{q+1}\sigma$, so we can exploit the frequency separation between $\nabla a_\xi^2$ and $\phi_\xi^2\varphi_{\xi_1}^2$ and gain a factor of $\lambda_{q+1}\sigma$ from the inverse divergence operator. More precisely, we recall the following lemma, which is a variant of Lemma B.1 of \cite{bucknonuniqueness}:
	\begin{lem}\label{gain}
		Let $a\in C^2(\mathbb{T}^3)$. For $1<p<\infty$, and for any $f\in L^p(\mathbb{T}^3)$, we have
		\begin{equation*}
			\||\nabla|^{-1}\mathbb{P}_{\neq 0}(a\mathbb{P}_{\geq k}f)\|_{L^p}\lesssim k^{-1}\|a\|_{C^2}\|f\|_{L^p}.
		\end{equation*}
	\end{lem}
	Using \eqref{Ei1} and \eqref{partial t et}, we have
	\begin{align*}
		&\operatorname{div}E_{i,1}+\partial_te_{q+1}^{i,t}\\
		&=\sum_{\xi\in \Lambda_{i}}\mathbb{P}_{\neq 0}\left(\nabla a_\xi^2\mathbb{P}_{ \geq \lambda_{q+1}\sigma}(\phi_{\xi}^{2}\varphi_{\xi_1}^{2})(\xi_{2}\otimes\xi-\xi\otimes\xi_{2})\right)-\frac{1}{\mu}\sum_{\xi\in\Lambda_i}\mathbb{P}_{\neq 0}\left(\partial_ta_\xi^2\mathbb{P}_{\geq \lambda_{q+1}\sigma}(\phi_\xi^2\varphi_{\xi_1}^2)\xi_2\right) \\
		&\quad\quad+\frac{1}{\mu}\sum_{\xi\in \Lambda_{i}}\nabla\Delta^{-1}\operatorname{div}\partial_t(a_\xi^2\mathbb{P}_{ \geq \lambda_{q+1}\sigma}(\phi_{\xi}^{2}\varphi_{\xi_1}^{2})\xi_2).
	\end{align*}
	Observe that
	\begin{align*}
		&\nabla\Delta^{-1}\operatorname{div}\partial_t(a_\xi^2\mathbb{P}_{ \geq \lambda_{q+1}\sigma}(\phi_{\xi}^{2}\varphi_{\xi_1}^{2})\xi_2)\\
		&=\nabla\Delta^{-1}\operatorname{div}(\partial_ta_\xi^2\mathbb{P}_{ \geq \lambda_{q+1}\sigma}(\phi_{\xi}^{2}\varphi_{\xi_1}^{2})\xi_2)
		+\nabla\Delta^{-1}\operatorname{div}(a_\xi^2\partial_t\mathbb{P}_{ \geq \lambda_{q+1}\sigma}(\phi_{\xi}^{2}\varphi_{\xi_1}^{2})\xi_2)\\
		&=\nabla\Delta^{-1}\operatorname{div}(\partial_ta_\xi^2\mathbb{P}_{ \geq \lambda_{q+1}\sigma}(\phi_{\xi}^{2}\varphi_{\xi_1}^{2})\xi_2)
		+\mu\nabla\Delta^{-1}\left(\xi_2\cdot\nabla a_\xi^2\mathbb{P}_{ \geq \lambda_{q+1}\sigma}(\xi\cdot\nabla\phi_{\xi}^{2}\varphi_{\xi_1}^{2})\right)\\
		&=\nabla\Delta^{-1}\operatorname{div}(\partial_ta_\xi^2\mathbb{P}_{ \geq \lambda_{q+1}\sigma}(\phi_{\xi}^{2}\varphi_{\xi_1}^{2})\xi_2)
		+\mu\nabla\Delta^{-1}\left(\xi\cdot\nabla  \left(\xi_2\cdot\nabla a_\xi^2\mathbb{P}_{ \geq \lambda_{q+1}\sigma}(\phi_{\xi}^{2}\varphi_{\xi_1}^{2})\right)\right)\\
		&\quad-\mu\nabla\Delta^{-1} \left(\left(\xi\cdot\nabla (\xi_2\cdot\nabla a_\xi^2
		)\right)\mathbb{P}_{ \geq \lambda_{q+1}\sigma}(\phi_{\xi}^{2}\varphi_{\xi_1}^{2})\right).
	\end{align*}
	Hence we get
	\begin{equation}\label{osci}
		\begin{aligned}
			&\|\mathcal{R}^F(\operatorname{div}E_{i,1}+\partial_te_{q+1}^{i,t})\|_{L^p}\\
			&\lesssim\sum_{\xi\in \Lambda_{i}}\left\{\left\||\nabla|^{-1}\left(\nabla a_\xi^2\mathbb{P}_{ \geq \lambda_{q+1}\sigma}(\phi_{\xi}^{2}\varphi_{\xi_1}^{2})\right)\right\|_{L^p}+\frac{1}{\mu}\left\||\nabla|^{-1}\left(\partial_ta_\xi^2\mathbb{P}_{\geq \lambda_{q+1}\sigma}(\phi_\xi^2\varphi_{\xi_1}^2)\right)\right\|_{L^p} \right.\\
			&\quad\quad\quad\quad+\frac{1}{\mu}\left\||\nabla|^{-1}\nabla\Delta^{-1}\operatorname{div}(\partial_ta_\xi^2\mathbb{P}_{ \geq \lambda_{q+1}\sigma}(\phi_{\xi}^{2}\varphi_{\xi_1}^{2}))\right\|_{L^p}\\
			&\quad\quad\quad\quad+\left\||\nabla|^{-1}\nabla\Delta^{-1}\left(\xi\cdot\nabla  \left(\xi_2\cdot\nabla a_\xi^2\mathbb{P}_{ \geq \lambda_{q+1}\sigma}(\phi_{\xi}^{2}\varphi_{\xi_1}^{2})\right)\right)\right\|_{L^p}\\
			&\quad\quad\quad\quad+\left.\left\||\nabla|^{-1}\nabla\Delta^{-1} \left(\left(\xi\cdot\nabla (\xi_2\cdot\nabla a_\xi^2
			)\right)\mathbb{P}_{ \geq \lambda_{q+1}\sigma}(\phi_{\xi}^{2}\varphi_{\xi_1}^{2})\right)\right\|_{L^p}\right\}\\
			&\lesssim\sum_{\xi\in \Lambda_{i}}\left\{\left\||\nabla|^{-1}\left(\nabla a_\xi^2\mathbb{P}_{ \geq \lambda_{q+1}\sigma}(\phi_{\xi}^{2}\varphi_{\xi_1}^{2})\right)\right\|_{L^p}+\frac{1}{\mu}\left\||\nabla|^{-1}\left(\partial_ta_\xi^2\mathbb{P}_{\geq \lambda_{q+1}\sigma}(\phi_\xi^2\varphi_{\xi_1}^2)\right)\right\|_{L^p} \right.\\
			&\quad\quad\quad\quad+\left\||\nabla|^{-1} \left(\xi_2\cdot\nabla a_\xi^2\mathbb{P}_{ \geq \lambda_{q+1}\sigma}(\phi_{\xi}^{2}\varphi_{\xi_1}^{2})\right)\right\|_{L^p}+\left.\left\||\nabla|^{-1} \left(\left(\xi\cdot\nabla (\xi_2\cdot\nabla a_\xi^2
			)\right)\mathbb{P}_{ \geq \lambda_{q+1}\sigma}(\phi_{\xi}^{2}\varphi_{\xi_1}^{2})\right)\right\|_{L^p}\right\},
		\end{aligned}
	\end{equation}
	where we used the fact that $\|\nabla\Delta^{-1}\operatorname{div}\|_{L^p\rightarrow L^p}+\|\nabla\Delta^{-1}\xi\cdot\nabla\|_{L^p\rightarrow L^p}+\|\nabla\Delta^{-1}\mathbb{P}_{\neq 0}\|_{L^p\rightarrow L^p}\lesssim 1$.
	For the first term, we can apply Lemma \ref{gain} with $a=\nabla a_\xi^2, f=\phi_\xi^2\varphi_{\xi_1}^2, k=\lambda_{q+1}\sigma$ and obtain
	\begin{equation}\label{osci2}
		\begin{aligned}
			\left\||\nabla|^{-1}\left(\nabla a_{\xi}^{2} \mathbb{P}_{ \geq \lambda_{q+1}\sigma}\left(\phi_{\xi}^{2}\varphi_{\xi_1}^2\right)\right)\right\|_{L^p}
			\lesssim&(\lambda_{q+1}\sigma)^{-1}\ell^{-9}(\sigma r)^{\frac{1}{p}-1}.
		\end{aligned}
	\end{equation}
	Other terms in \eqref{osci} can be estimated similarly as the first term.
	In view of \eqref{osc1} and \eqref{osci2}, we can conclude that
	\begin{equation}\label{estimate osci}
		\|{R}_{osc}^{i}\|_{L^1}\lesssim\delta_{q+1}\frac{\sigma}{r}+(\lambda_{q+1}\sigma)^{-1}\ell^{-9}(\sigma r)^{\frac{1}{p}-1}.
	\end{equation}
	
	Next we estimate the Reynolds oscillation stress $R_{osc}^v$. By \eqref{defRF}, \eqref{cancelRv} and the fact that $\xi_2\cdot \nabla (\phi_\xi^2\varphi_{\xi_1}^2)=0$, we can write
	\begin{equation*}
		\begin{aligned}
			&\operatorname{div}(w_{q+1}^{p}\otimes w_{q+1}^{p}-e_{q+1}^{i,p}\otimes e_{q+1}^{i,p}+R^v_\ell )\\
			&=\operatorname{div}\left(\sum_{\xi\in \Lambda_v}a_{\xi}^{2}\phi_{\xi}^{2}\varphi_{\xi_1}^{2}(\xi\otimes\xi)
			+\sum_{i=1}^3\sum_{\xi \in \Lambda_i}a_\xi^2\phi_{\xi}^{2}\varphi_{\xi_1}^{2}(\xi\otimes\xi-\xi_2\otimes\xi_2)+R^v_\ell\right)\\
			&\quad\quad\quad+\operatorname{div}\left(\sum_{\xi\neq\xi' \in \Lambda} a_{\xi} a_{\xi'} \phi_{\xi} \varphi_{\xi_1} \phi_{\xi'}\varphi_{\xi_1'}(\xi \otimes \xi')
			-\sum_{i=1}^3\sum_{\xi\neq\xi' \in \Lambda_i}a_{\xi} a_{\xi'} \phi_{\xi} \varphi_{\xi_1} \phi_{\xi'}\varphi_{\xi_1'}(\xi_2 \otimes \xi'_2)\right)\\
			&=\operatorname{div}\left(\sum_{\xi\in \Lambda_v}a_{\xi}^{2}\phi_{\xi}^{2}\varphi_{\xi_1}^{2}(\xi\otimes\xi)+R^v_\ell+R_F\right)
			+\operatorname{div}\left(\sum_{i=1}^3\sum_{\xi\in \Lambda_i}a_{\xi}^{2}\mathbb{P}_{\geq \lambda_{q+1}\sigma}(\phi_\xi^2\varphi_{\xi_1}^2)(\xi\otimes\xi-\xi_2\otimes\xi_2)\right)\\
			&\quad\quad\quad+\operatorname{div}\left(\sum_{\xi\neq\xi' \in \Lambda} a_{\xi} a_{\xi'} \phi_{\xi} \varphi_{\xi_1} \phi_{\xi'}\varphi_{\xi_1'}(\xi \otimes \xi')
			-\sum_{i=1}^3\sum_{\xi\neq\xi' \in \Lambda_i}a_{\xi} a_{\xi'} \phi_{\xi} \varphi_{\xi_1} \phi_{\xi'}\varphi_{\xi_1'}(\xi_2 \otimes \xi'_2)\right)\\
			&=\operatorname{div}\left(\sum_{\xi\in \Lambda}a_{\xi}^{2}\mathbb{P}_{\geq \lambda_{q+1}\sigma}(\phi_\xi^2\varphi_{\xi_1}^2)(\xi\otimes\xi)\right)+\nabla\rho_v
			-\sum_{i=1}^3\sum_{\xi\in \Lambda_i}\xi_2\cdot\nabla a_{\xi}^{2}\mathbb{P}_{\geq \lambda_{q+1}\sigma}(\phi_\xi^2\varphi_{\xi_1}^2)\xi_2\\
			&\quad\quad\quad+\operatorname{div}\left(\sum_{\xi\neq\xi' \in \Lambda} a_{\xi} a_{\xi'} \phi_{\xi} \varphi_{\xi_1} \phi_{\xi'}\varphi_{\xi_1'}(\xi \otimes \xi')
			-\sum_{i=1}^3\sum_{\xi\neq\xi' \in \Lambda_i}a_{\xi} a_{\xi'} \phi_{\xi} \varphi_{\xi_1} \phi_{\xi'}\varphi_{\xi_1'}(\xi_2 \otimes \xi'_2)\right).
		\end{aligned}
	\end{equation*}
	We put $\nabla\rho_v$ in the pressure $\nabla p_{osc}$. 
	The last two terms can be estimated similarly as \eqref{osc1} by Lemma \ref{gain}. Now we consider the first term. In view of \eqref{partial t wt} we obtain
	\begin{equation*}
		\begin{aligned}
			&\operatorname{div}\left(\sum_{\xi\in \Lambda}a_{\xi}^{2}\mathbb{P}_{\geq \lambda_{q+1}\sigma}(\phi_\xi^2\varphi_{\xi_1}^2)(\xi\otimes\xi)\right)+\partial_t w_{q+1}^t\\
			&=\sum_{\xi\in\Lambda}\nabla a_\xi^2\mathbb{P}_{\geq \lambda_{q+1}\sigma}(\phi_\xi^2\varphi_{\xi_1}^2)\xi\otimes\xi
			-\frac{1}{\mu}\sum_{\xi\in\Lambda}\partial_ta_\xi^2\mathbb{P}_{\geq \lambda_{q+1}\sigma}(\phi_\xi^2\varphi_{\xi_1}^2)\xi+\frac{1}{\mu}\sum_{\xi\in\Lambda}\nabla\left(\Delta^{-1}\operatorname{div}\partial_t(a_\xi^2\phi_\xi^2\varphi_{\xi_1}^2\xi)\right).
		\end{aligned}
	\end{equation*}
	We also apply Lemma \ref{gain} to estimate the first two terms, and classify the last term into $\nabla p_{osc}$.
	The oscillation error $R_{osc}^v$ is given by
	\begin{equation}\label{Rvosc}
		\begin{aligned}
			R_{osc}^v&=\mathcal{R}\left(\sum_{\xi\in\Lambda}\nabla a_\xi^2\mathbb{P}_{\geq \lambda_{q+1}\sigma}(\phi_\xi^2\varphi_{\xi_1}^2)\xi\otimes\xi
			-\frac{1}{\mu}\sum_{\xi\in\Lambda}\partial_ta_\xi^2\mathbb{P}_{\geq \lambda_{q+1}\sigma}(\phi_\xi^2\varphi_{\xi_1}^2)\xi\right)\\
			&\quad\quad-\mathcal{R}\left(\sum_{i=1}^3\sum_{\xi\in \Lambda_i}\xi_2\cdot\nabla a_{\xi}^{2}\mathbb{P}_{\geq \lambda_{q+1}\sigma}(\phi_\xi^2\varphi_{\xi_1}^2)\xi_2\right)+\sum_{\xi\neq\xi' \in \Lambda} a_{\xi} a_{\xi'} \phi_{\xi} \varphi_{\xi_1} \phi_{\xi'}\varphi_{\xi_1'}(\xi \hat{\otimes} \xi')\\
			&\quad\quad-\sum_{i=1}^3\sum_{\xi\neq\xi' \in \Lambda_i}a_{\xi} a_{\xi'} \phi_{\xi} \varphi_{\xi_1} \phi_{\xi'}\varphi_{\xi_1'}(\xi_2 \hat{\otimes} \xi'_2).
		\end{aligned}
	\end{equation}
	Hence we can conclude
	\begin{equation}\label{estimate oscv}
		\|{R}_{osc}^{v}\|_{L^1}\lesssim\delta_{q+1}\frac{\sigma}{r}+(\lambda_{q+1}\sigma)^{-1}\ell^{-15}(\sigma r)^{\frac{1}{p}-1}.
	\end{equation}
	
	
	
	\subsection{Verification of Inductive Estimates for Stresses}
	In this section, we verify inductive estimate \eqref{STRESS L1} for the stresses. By \eqref{commutator}, \eqref{estimate linv}-\eqref{estimate corri}, and \eqref{estimate osci}-\eqref{estimate oscv} we can conclude
	\begin{equation*}
		\begin{aligned}
			&\|R_{q+1}^i\|_{L^1}+\|R_{q+1}^v\|_{L^1}\\
			&\leq \|R_{lin}^i\|_{L^1}+\|R_{osc}^i\|_{L^1}+\|R_{corr}^i\|_{L^1}+\|R_{comm}^i\|_{L^1}+\|R_{lin}^v\|_{L^1}+\|R_{osc}^v\|_{L^1}+\|R_{corr}^v\|_{L^1}+\|R_{comm}^v\|_{L^1}\\
			&\lesssim\ell^{-15}(\sigma r)^{\frac{1}{p}-1}\left(\lambda_{q+1}^{\theta_*}(\sigma r)^\frac{1}{2}+\mu\frac{\sigma^\frac{3}{2}} {r^\frac{1}{2}}+(\lambda_{q+1}\sigma)^{-1}\right)+\delta_{q+1}\frac{\sigma}{r}+\ell^{-6}\mu^{-1}(\sigma r)^{-\frac{1}{2}}\delta_{q+1}^{\frac{1}{2}}+\ell^2\lambda_q^8,
		\end{aligned}
	\end{equation*}
	holds for $p>1$. By \eqref{defpara}, \eqref{para def} and \eqref{defell}, we obtain 
	\begin{align*}
		&\|R_{q+1}^i\|_{L^1}+\|R_{q+1}^v\|_{L^1}\\
		&\lesssim \lambda_{q+1}^{\frac{300}{b}-(2-8\alpha)(\frac{1}{p}-1)}\left(\lambda_{q+1}^{\theta_*-1+4\alpha}+\lambda_{q+1}^{-\alpha}+\lambda_{q+1}^{-2\alpha}\right)+\lambda_1^{3\beta}\lambda_{q+1}^{-4\alpha-2\beta}+\lambda_1^{\frac{3\beta}{2}}\lambda_{q+1}^{\frac{120}{b}-3\alpha-\beta}+\lambda_{q+1}^{-\frac{32}{b}}.
	\end{align*}
	We choose $p$ such that
	\begin{equation*}
		\left(\frac{1}{p}-1\right)(8\alpha-2)=\frac{\alpha}{2},
	\end{equation*}
	namely,
	\begin{equation*}
		p=\frac{4-16\alpha}{4-17\alpha}\in(1,2).
	\end{equation*}
	By \eqref{theta}-\eqref{para ineq}, one has $$\theta_*-1+4\alpha\leq-4\alpha,$$ 
	and 
	\begin{align*}
		\|R_{q+1}^i\|_{L^1}+\|R_{q+1}^v\|_{L^1} & \lesssim \lambda_{q+1}^{\frac{300}{b}-\frac{\alpha}{2}}+\lambda_1^{3\beta}\lambda_{q+1}^{-4\alpha-2\beta}+\lambda_1^{\frac{3\beta}{2}}\lambda_{q+1}^{\frac{120}{b}-3\alpha-\beta}+\lambda_{q+1}^{-\frac{32}{b}}\\
		&\leq \lambda_1^{3\beta}\lambda_{q+1}^{-2\beta b} =\delta_{q+2},
	\end{align*}
	where we have used the fact
	\begin{align*}
		\max\left\{\frac{300}{b}-\frac{\alpha}{2},\  -4\alpha,\  \frac{120}{b}-3\alpha,\  -\frac{32}{b}\right\} = -\frac{32}{b} \leq -2\beta b.
	\end{align*}
	Moreover, by Lemma \ref{lemma estimate of a} and Lemma \ref{philemma} we obtain
	\begin{equation*}
		\|R_{q+1}^i\|_{C^1_{t,x}}+\|R_{q+1}^v\|_{C^1_{t,x}}\lesssim\sum_{\xi\in\Lambda}\|a_\xi^2\|_{C^3_{t,x}}\|\phi_\xi^2\|_{C^3_{t,x}}\|\varphi_{\xi_1}^2\|_{C^3_{t,x}}\leq\lambda_{q+1}^{10}.
	\end{equation*}
	This yields the inductive estimate \eqref{STRESS C1}.
	Finally, by  \eqref{suppa} and the definition of the new stresses (see \eqref{Rcomm}, \eqref{Rlico}, \eqref{Riosc} and \eqref{Rvosc}), we have 
	\begin{align*}
		\operatorname{supp}_t (R^v_{q+1},R^1_{q+1},R^2_{q+1},R^3_{q+1})&\subseteq \left(\bigcup_{\xi\in\Lambda}\operatorname{supp}_ta_\xi\right)\cup\operatorname{supp}_t(v_{\ell},F^1_\ell,F^2_{\ell},F^3_{\ell})\\
		&\subseteq \bigcup_{i=1}^3O_{\delta_{q+2}}(\operatorname{supp}_t(v_q,R^v_q,F^i_q,R^i_q)).
	\end{align*}
	Combining this with \eqref{incresupp}, we obtain \eqref{itersupp}.
	This completes the proof of Proposition \ref{main iteration}.
	\section{Acknowledgement}
	The authors were supported by NSFC (grant No.11725102). The authors are grateful to Professor Zhen Lei for helpful suggestions and discussions. %

\end{document}